\documentclass[final]{amsart}
\usepackage{amssymb}

\usepackage{xy}
\xyoption{all}

\newtheorem{theorem}{Theorem}[section]
\newtheorem{lemma}[theorem]{Lemma}
\newtheorem{proposition}[theorem]{Proposition}
\newtheorem{corollary}[theorem]{Corollary}

\newtheorem{problem}[theorem]{Problem}
\theoremstyle{definition}

\newcommand{\F}{\mathcal {F}}
\newcommand{\Ra}{\Rightarrow}
\newcommand{\U}{\mathcal U}
\newcommand{\A}{\mathcal A}
\newcommand{\BB}{\mathcal B}
\newcommand{\w}{\omega}
\newcommand{\IN}{\mathbb N}
\newcommand{\ostar}{\circledast}
\newcommand{\la}{\langle}
\newcommand{\ra}{\rangle}

\newcommand{\IZ}{\mathbb Z}

\title[Characterizing semigroups with commutative extensions]{Characterizing semigroups $X$ with commutative extensions\\ $\varphi(X)$, $\lambda(X)$, $N_2(X)$, $\upsilon(X)$}
\author[T. Banakh, V. Gavrylkiv]{Taras Banakh and Volodymyr Gavrylkiv}
\address[T.~Banakh]{Ivan Franko University of Lviv, Ukraine and
Jan Kochanowski University, Kielce, Poland}
\email{t.o.banakh@gmail.com}
\address[V.~Gavrylkiv]{Vasyl Stefanyk Precarpathian National University,
Ivano-Frankivsk,
Ukraine}
\email{vgavrylkiv@yahoo.com}
\thanks{The first author has been partially financed by NCN means granted by decision DEC-2011/01/B/ST1/01439}
\subjclass{20M10, 20M14, 20M17, 20M18, 54B20}
\keywords{Commutative semigroup, superextension, semigroup of filters, semigroup of linked upfamilies}

\begin{document}

\begin{abstract}
We characterize semigroups $X$ whose semigroups of  filters $\varphi(X)$, maximal linked systems
$\lambda(X)$, linked upfamilies $N_2(X)$,
and upfamilies $\upsilon(X)$ are commutative.
\end{abstract}
\maketitle

\section{Introduction}

In this paper we  investigate  the algebraic structure of various
extensions of an inverse semigroup  $X$ and detect semigroups
whose extensions $\varphi(X)$, $\lambda(X)$, $N_2(X)$,
$\upsilon(X)$ and their subsemigroups  $\varphi^\bullet(X)$,
$\lambda^\bullet(X)$, $N_2^\bullet(X)$, $\upsilon^\bullet(X)$ are
commutative.

The thorough study of various extensions of semigroups was started in
\cite{G2} and continued in \cite{BG2}--\cite{BGN}. The
largest among these extensions is the semigroup $\upsilon(X)$ of
all upfamilies on $X$.

A family $\F$ of subsets of a set $X$ is called an {\em upfamily}
if $\emptyset\notin\F\ne\emptyset$ and for each set $F\in\F$ any
subset $E\supset F$ of $X$ belongs to $\F$.
Each family $\F$ of non-empty subsets of $X$ generates the upfamily
$$\la F:F\in\F\ra=\{E\subset X:\exists F\in \F\;\;F\subset E\}.$$

The space of all
upfamilies on $X$ is denoted by $\upsilon(X)$. It is a closed
subspace of the double power-set $\mathcal P(\mathcal P(X))$
endowed with the compact Hausdorff topology of the Tychonoff
product $\{0,1\}^{\mathcal P(X)}$. Identifying each point $x\in X$
with the {\em principal ultrafilter} $\la x\ra=\{A\subset X:x\in A\}$, we can
identify $X$ with a subspace of $\upsilon(X)$. Because of that we
call $\upsilon(X)$ an extension of $X$. For an upfamily $\F\in\upsilon(X)$ by
$$\F^\perp=\{E\subset X:\forall F\in\F\;\;E\cap F\ne\emptyset\}$$we denote
the {\em transversal} of $\F$. By \cite{G1},
$(\F^\perp)^\perp=\F$, so
$$\perp:\upsilon(X)\to\upsilon(X),\;\;\perp:\F\mapsto\F^\perp,$$
is an involution on $\upsilon(X)$. For a subset $S\subset
\upsilon(X)$ we put $S^\perp=\{\F^\perp:\F\in
S\}\subset\upsilon(X)$.

The compact Hausdorff space $\upsilon(X)$ contains many other
important extensions of $X$ as closed subspaces. In particular, it
contains the spaces $N_2(X)$ of linked upfamilies, $\lambda(X)$ of
maximal linked upfamilies, $\varphi(X)$ of filters, and  $\beta(X)$ of
ultrafilters on $X$; see \cite{G1}. Let us recall that an upfamily
$\F\in\upsilon(X)$ is called
\begin{itemize}
\item {\em linked} if $A\cap B\ne \emptyset$ for any sets
$A,B\in\F$; \item {\em maximal linked} if $\F=\F'$ for any linked
upfamily $\F'\in\upsilon(X)$ that contains $\F$; \item a {\em
filter} if $A\cap B\in\F$ for any $A,B\in\F$; \item an {\em
ultrafilter} if  $\F=\F'$ for any filter $\F'\in\upsilon(X)$ that
contains $\F$.
\end{itemize}
The family $\beta(X)$ of all ultrafilters on $X$ is called the
{\em Stone-\v Cech extension} and the family $\lambda(X)$ of all
maximal linked upfamilies is called the {\em superextension} of
$X$, see \cite{vM} and \cite{Ve}. It can be shown that $\lambda(X)=\{\F\in\upsilon(X):\F^\perp=\F\}$, so $\lambda(X)^\perp=\lambda(X)$ and $\beta(X)^\perp=\beta(X)$. The arrows in the following diagram
denote the identity inclusions between various extensions of a set
$X$.
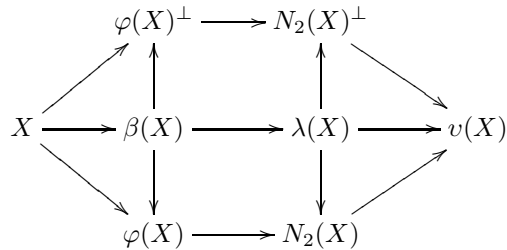
$$\xymatrix{
&\varphi(X)^\perp\ar[r]&N_2(X)^\perp\ar[rd]&\\
X\ar[r]\ar[ru]\ar[rd]&\beta(X)\ar[r]\ar[d]\ar[u]&\lambda(X)\ar[r]\ar[d]\ar[u]&\upsilon(X)\\
 &\varphi(X)\ar[r]&N_2(X)\ar[ur].}
$$

We say that an upfamily $\U\in\upsilon(X)$ {\em is finitely
supported} if $\U=\la F_1,\dots,F_n\ra$ for some non-empty finite subsets $F_1,\dots,F_n\subset X$. By $\upsilon^\bullet(X)$ we
denote the subspace of $\upsilon(X)$ consisting of finitely
supported upfamilies on $X$. Let
$$\varphi^\bullet(X)=\varphi(X)\cap\upsilon^\bullet(X), \;\;
\lambda^\bullet(X)=\lambda(X)\cap\upsilon^\bullet(X),\;\;
N_2^\bullet(X)=N_2(X)\cap\upsilon^\bullet(X).$$ Since each
finitely supported ultrafilter is principal, the set
$\beta^\bullet(X)=\beta(X)\cap\upsilon^\bullet(X)$ coincides with
$X$ (identified with the subspace of all principal ultrafilters in
$\upsilon(X)$~). The embedding relations between these spaces of
finitely supported upfamilies are indicated in the following
diagram:

$$\xymatrix{
&\varphi^\bullet(X)^\perp\ar[r]&N_2^\bullet(X)^\perp\ar[rd]&\\
X\ar@{=}[r]\ar[ru]\ar[rd]&\beta^\bullet(X)\ar[r]\ar[d]\ar[u]&\lambda^\bullet(X)\ar[r]\ar[d]\ar[u]&\upsilon^\bullet(X)\\
 &\varphi^\bullet(X)\ar[r]&N_2^\bullet(X)\ar[ur].}
$$

Any map $f:X\to Y$ induces a continuous map $$\upsilon
f:\upsilon(X)\to\upsilon(Y),\quad \upsilon f:\F\mapsto \{A\subset
Y:f^{-1}(A)\in\F\}.$$ By \cite{G1}, $\upsilon
f(\F^\perp)=\big(\upsilon f(\F)\big)^\perp$ and $\upsilon
f\big(\beta(X)\big)\subset\beta(Y)$, $\upsilon
f\big(\lambda(X)\big)\subset\lambda(Y)$, $\upsilon
f\big(\varphi(X)\big)\subset\varphi(Y)$, $\upsilon
f\big(N_2(X)\big)\subset N_2(Y)$, and $\upsilon
f\big(\upsilon^\bullet(X)\big)\subset\upsilon^\bullet(X)$. If the
map $f$ is injective, then  $\upsilon f$ is a topological
embedding, which allows us to identify the extensions $\beta(X)$,
$\lambda(X)$, $\varphi(X)$, $N_2(X)$ and $\upsilon(X)$ with
corresponding closed subspaces in $\beta(Y)$, $\lambda(Y)$,
$\varphi(Y)$, $N_2(Y)$ and $\upsilon(Y)$, respectively.

If $*:X\times X\to X$, $*:(x,y)\mapsto xy$, is a binary operation
on $X$, then there is an obvious way of extending this operation
onto the space $\upsilon(X)$. Just put
$$\A\circledast\BB=\la A*B:A\in\A,\;B\in\BB\ra$$where $A*B=\{ab:a\in A,\;b\in B\}$
is the pointwise product of sets $A,B\subset X$. The upfamily
$\A\circledast\BB$ will be called the {\em pointwise product} of
the upfamilies $\A,\BB$. It is clear that this extension
$\circledast:\upsilon(X)\times\upsilon(X)\to\upsilon(X)$ of the
operation $*:X\times X\to X$ is associative and commutative if so
is the operation $*$. So, for an associative binary operation $*$
on $X$, its extension $\upsilon(X)$ endowed with the operation
$\circledast$ of pointwise product becomes a semigroup, containing
the subspaces $\varphi(X)$, $\varphi^\bullet(X)$, $N_2(X)$, and
$N_2^\bullet(X)$ as subsemigroups. However, the subspaces
$\beta(X)$ and $\lambda(X)$ are not subsemigroups in
$(\upsilon(X),\circledast)$. To make $\beta(X)$ a semigroup
extension of $X$, Ellis \cite{Ellis} suggested a less obvious
extension of an (associative) binary operation $*:X\times X\to X$
to an (associative) binary operation on $\beta(X)$ letting
\begin{equation}\label{formula}\mathcal A*\mathcal B=\Big\la\bigcup_{a\in A}a*B_a:A\in\A,\;\;\{B_a\}_{a\in A}\subset\mathcal B\Big\ra
\end{equation}for ultrafilters $\A,\mathcal B\in\beta(X)$. It is clear that $\A\ostar\BB\subset\A*\BB$ but in general $\A\ostar\BB\ne\A*\BB$.

In \cite{G2} it was observed that the formula (\ref{formula})
determines an extension of the operation $*$ to an (associative)
binary operation $*:\upsilon(X)\times\upsilon(X)\to\upsilon(X)$ on
the extension $\upsilon(X)$ of $X$. So, for each semigroup
$(X,*)$, its extension $\upsilon(X)$ endowed with the extended
operation $*$ is a semigroup, containing the subspaces $\beta(X)$,
$\varphi(X)$, $\lambda(X)$, $N_2(X)$ as closed subsemigroups.
Moreover, since $\upsilon^\bullet(X)$ is a subsemigroup of
$\upsilon(X)$, the subspaces $X=\beta^\bullet(X)$,
$\varphi^\bullet(X)$, $\lambda^\bullet(X)$, $N_2^\bullet(X)$ also
are subsemigroups of $\upsilon(X)$. Algebraic and topological
properties of these semigroups have been studied in \cite{G2},
\cite{BG2}--\cite{BGN}. In particular, in \cite{BGN} and \cite{BG10} we studied
properties of extensions of groups, while \cite{BGs} and
\cite{BGi} were devoted to extensions of semilattices and inverse
semigroups, respectively. In contrast to the operation
$\circledast$, the extended operation $*$ on the semigroup
$\upsilon(X)$ and its subsemigroups rarely is commutative. For
example, by \cite{BGN} a group $X$ has commutative superextension
$\lambda(X)$ if and only if $X$ is a group of cardinality $|X|\le
4$. According to \cite{BGs}, a semilattice $X$ has commutative
extension $\upsilon(X)$ if and only if $X$ is finite and linearly
ordered.

Let $X$ be a semigroup. A subsemigroup $S\subset\upsilon(X)$ is
defined to be {\em supercommutative} if $$\A*\mathcal B=\A\ostar
\mathcal B=\mathcal B\ostar \A=\mathcal B*\A$$for any upfamilies
$\A,\mathcal B\in S$. It is clear that each supercommutative
subsemigroup $S\subset\upsilon(X)$ is commutative. The converse is
not true as we shall see in Section~\ref{s10}.

In this paper we study the commutativity and supercommutativity of the semigroups
$\upsilon(X)$, $N_2(X)$, $\lambda(X)$, $\varphi(X)$, $\beta(X)$,
$\upsilon^\bullet(X)$, $N_2^\bullet(X)$, $\lambda^\bullet(X)$,
$\varphi^\bullet(X)$ and characterize semigroups
$X$ whose various extensions are commutative or supercommutative. In the preliminary
Sections~\ref{s2}, \ref{s3} we shall analyze the structure of periodic
commutative semigroups and projective extensions of semigroup, Section~\ref{s5} is devoted to
square-linear semigroups which will play a crucial role in
Sections~\ref{s7} and \ref{s8} devoted to the study of
commutativity and supercommutativity of the semigroups $\upsilon(X)$,
$\upsilon^\bullet(X)$, $N_2(X)$ and $N_2^\bullet(X)$. In
Section~\ref{s6} we characterize semigroups $X$ with
(super)commutative extensions $\beta(X)$, $\varphi(X)$,
$\varphi^\bullet(X)$, and in Section~\ref{s9} we detect semigroups with
commutative extensions $\lambda(X)$ and $\lambda^\bullet(X)$. In Section~\ref{s10} we study the structure of semigroups $X$ whose superextension $\lambda(X)$ is supercommutative.

\section{The structure of periodic commutative semigroups}\label{s2}

In this section we recall some known information on the structure
of periodic commutative semigroups. A semigroup $S$ is called {\em
periodic} if each element $x\in S$ generates a finite semigroup
$\{x^k\}_{k\in\IN}$. A semigroup $S$ generated by a single element
$x$ is called {\em monogenic}. If a monogenic semigroup is
infinite, then it is isomorphic to the additive semigroup $\IN$ of positive integers.
A finite monogenic semigroup $S=\{x^k\}_{k\in\IN}$ also has
simple structure (cf. \cite[\S1.2]{Howie}): there are positive
integer numbers $n<m$ such that
\begin{itemize}
\item $S=\{x,\dots,x^{m-1}\}$, $m=|S|+1$ and $x^n=x^m$;
\item $C_x=\{x^n,\dots,x^{m-1}\}$ is
a cyclic subgroup of $S$;
\item the cyclic subgroup $C_x$
coincides with the minimal ideal of $S$;
\item the neutral element
$e_x$ of the group $C_x$ is a unique idempotent of $S$ and the
cyclic group $C_x$ is generated by the element $xe_x$.
\end{itemize}
Such monogenic semigroups will be denoted by $\la x\mid x^n=x^m\ra$.
\smallskip

For a semigroup $S$ let $$E(S)=\{e\in S:ee=e\}$$ be the {\em idempotent part} of $S$.
For each idempotent $e\in E(S)$ let $$H_e=\{x\in S:\exists y\in S\;\;xyx=x,\;yxy=y,\;xy=e=yx\}$$be the maximal subgroup of $S$ containing the idempotent $e$. The union $$C(S)=\bigcup_{e\in E(S)}H_e$$ of all (maximal) subgroups of $S$ is called the {\em Clifford part} of $S$. The Clifford part $C(S)$ is contained in the {\em regular part}
$$R(S)=\{x\in S:x\in xSx\}$$of $S$.
If a semigroup $S$ is commutative, then $R(S)=C(S)$ and the subsets $E(S)$ and $R(S)=C(S)$ are subsemigroups of $S$.

If a semigroup $S$ is periodic, then for each element $x\in S$ the
monogenic semigroup $\{x^k\}_{k\in\IN}$ contains a unique
idempotent $e_x$. So, we can consider the map
$$e_*:S\to E(S),\;\;e_*:x\mapsto e_x,$$which
projects the semigroup $S$ onto its idempotent part $E(S)$. The
map $$c_*:S\to C(S),\;\;c_*:x\mapsto e_x\cdot x,$$projects the
semigroup $S$ onto its Clifford part. If a periodic semigroup $S$
is commutative, then the projections $e_*:S\to E(S)$ and $c_*:S\to
C(S)$ are semigroup homomorphisms. In this case, for every
idempotent $e\in E(S)$, $S_e=\{x\in S:e_x=e\}$ is a subsemigroup
of $S$ with a unique idempotent $e$. So, the semigroup $S$
decomposes into the disjoint union $S=\bigcup_{e\in E(S)}S_e$ of
semigroups $S_e$ parametrized by idempotents $e\in E(S)$.

\section{Projection Extensions of Semigroups}\label{s3}

A semigroup $X$ is called a {\em projection extension} of a subsemigroup $Z\subset X$ if there is a function $\pi:X\to Z$ (called the {\em projection of $X$ onto $Z$}) such that
\begin{itemize}
\item $\pi(z)=z$ for each $z\in Z$;
\item $x\cdot y=\pi(x)\cdot\pi(y)\in Z$ for all $x,y\in X$.
\end{itemize}
It follows from $\pi(xy)=xy=\pi(x)\cdot\pi(y)$ that the projection $\pi:X\to Z$ necessarily is a homomorphism of $X$ onto its subsemigroup $Z$.

If a map $\pi:X\to Z$ of semigroups $X$ and $Z$ is a homomorphism,
then by \cite{G2} the map $\upsilon\pi:\upsilon(X)\to\upsilon(Z)$
is a homomorphism  too. So, we have the following statement.

\begin{proposition}\label{p3.1} If a semigroup $X$ is a projection extension of a subsemigroup $Z\subset X$, then the projection $\pi:X\to Z$ induces a homomorphism $\upsilon\pi:\upsilon(X)\to\upsilon(Z)$ witnessing that the semigroup $\upsilon(X)$ is a projection extension of the subsemigroup $\upsilon(Z)$.
\end{proposition}

\begin{corollary} Assume that a semigroup $X$ is a projection extension of a subsemigroup $Z\subset X$, $\pi:X\to Z$ is a projection of $X$ onto $Z$. Then each subsemigroup $S\subset\upsilon(X)$ with $\upsilon\pi(S)\subset S$ is a projection extension of the subsemigroup $\upsilon\pi(S)=S\cap\upsilon(Z)$. Consequently, the semigroup $S$ is \textup{(}super\textup{)}commutative if and only if so is its subsemigroup $S\cap\upsilon(Z)$.
\end{corollary}

\begin{corollary}\label{c:pext2} Assume that a semigroup $X$ is a projection extension of a subsemigroup $Z\subset X$, and $\varepsilon\in\{\upsilon,\upsilon^\bullet,N_2,N_2^\bullet, \varphi,\varphi^\bullet,\lambda,\lambda^\bullet,\beta,\beta^\bullet\}$. The extension $\varepsilon(X)$ of $X$ is \textup{(}super\textup{)}commutative if and only if the extension $\varepsilon(Z)$ of the semigroup $Z$ is \textup{(}super\textup{)}commutative.
\end{corollary}

\section{Semicomplete digraphs}\label{s4}

In this section we recall some information on digraphs. In the next section this information will be used for describing the structure of square-linear semigroups.

By an {\em directed graph} (briefly, a {\em digraph}) we shall understand a pair $(X,\Delta)$ consisting of a set $X$ and a subset $\Delta\subset X\times X$. Elements $x\in X$ are called {\em vertices} and ordered pairs $(x,y)\in\Delta$ called {\em edges} of the digraph $(X,\Delta)$.
An edge $(x,y)\in\Delta$ is called {\em pure} if $(y,x)\notin\Delta$.
A digraph $(X,\Delta)$ is called
{\em complete} if $\Delta=X\times X$ and {\em semicomplete} if $\Delta\cup\Delta^{-1}=X\times X$, where $\Delta^{-1}=\{(y,x):(x,y)\in \Delta\}$.

A sequence $x_0,\dots,x_n$ of vertices of a digraph $(X,\Delta)$ is called a ({\em pure}) {\em cycle of length $n$} if $x_0=x_n$ and for every $i<n$ the pair $(x_i,x_{i+1})$ is a (pure) edge of the digraph $(X,\Delta)$. A cycle $x_0,x_1,\dots,x_n$ in a digraph $(X,D)$ is called {\em bipartite} if the number $n$ is even and for each numbers $i,j\in\{1,\dots,n\}$ with odd difference $i-j$ we get $(x_i,x_j)\notin\Delta\cap\Delta^{-1}$. Bipartite cycles can be equivalently defined as cycles $x_0,y_1,x_1,y_2,\dots,y_n,x_n$ such that $(x_i,y_j)\notin\Delta\cap\Delta^{-1}$ for any $1\le i,j\le n$.

It is easy to see that a cycle of length 4 is bipartite if and only if it is pure.

\begin{lemma}\label{l:graph} A semicomplete digraph $(X,\Delta)$ contains a pure cycle of length 4 if and only if it contains a bipartite cycle.
\end{lemma}

\begin{proof} Let $x_0,x_1,\dots,x_n$ be a bipartite cycle in the digraph of the smallest possible  length $n$. The length $n$ is even and cannot be equal to 2 as otherwise $(x_1,x_2)=(x_1,x_0)\in\Delta\cap\Delta^{-1}$. So, $n\ge 4$.

We claim that $n=4$. Assume conversely that $n>4$ and consider the pair $(x_0,x_3)$. Since the cycle is bipartite and the digraph $(X,\Delta)$ is semicomplete, either $(x_0,x_3)$ or $(x_3,x_0)$ is a pure edge of the digraph. If $(x_3,x_0)\in\Delta$, then $x_0,x_1,x_2,x_3,x_0$ is a bipartite (and pure) cycle of length 4 in $(X,\Delta)$. If $(x_0,x_3)\in\Delta$, then $x_0,x_3,x_5,\dots,x_n$ is a bipartite cycle of length $n-2\ge 4$ in $(X,\Delta)$, which contradicts the minimality of $n$.
\end{proof}

\section{Square-linear semigroups}\label{s5}

A semigroup $S$ is called {\em linear} if $xy\in\{x,y\}$ for any elements $x,y\in S$. It follows that each element $x$ of a linear semigroup is an idempotent. So, linear semigroups are {\em bands}, i.e., semigroups of idempotents. Commutative bands are called {\em semilattices}. So, each linear commutative semigroup is a semilattice. Each semilattice $E$ is endowed with a partial order $\le$ defined by $x\le y$ iff $xy=x$.

A semigroup $S$ is called {\em square-linear} if $xy\in\{x^2,y^2\}$ for all elements $x,y\in S$.

\begin{proposition}\label{p:square-linear} Let $S$ be a square-linear commutative semigroup and $x,y,z\in S$ be any elements. Then
\begin{enumerate}
\item $S$ is periodic and $x^3=x^4=e_x\in E(S)$;
\item the idempotent part $E(S)$ of $S$ is a linear semilattice;
\item the Clifford part $C(S)$ of $S$ coincides with $E(S)$;
\item $xy=e_xe_y$ if $x^2,y^2\in E(S)$;
\item $xyz=e_xe_ye_z$;
\item If $x^2\notin E(S)$, then $e_x$ is the largest element of the semilattice $E(S)$.
\end{enumerate}
\end{proposition}

\begin{proof}
1. It follows from $x^3=x\cdot x^2\in\{x^2,x^4\}$ that $x^3=x^4=e_x$ and hence $x^3=x^n$ for all $n\ge 3$. So, the monogenic semigroup $\{x^n\}_{n\in\IN}=\{x,x^2,x^3\}$ is finite and hence $S$ is periodic.
\smallskip

2. If $x,y$ are idempotents, then $xy\in\{x^2,y^2\}=\{x,y\}$ implies that the semilattice $E(S)$ is linear.
\smallskip

3. The identity $x^3=x^4$ implies that each subgroup of $S$ is trivial and hence $C(S)=E(S)$.
\smallskip

4. If $x^2,y^2\in E(S)$, then $x^4=x^2$ and hence $x^2=x^4=x^3=e_x$. Then $xy\in\{x^2,y^2\}=\{e_x,e_y\}$ implies that $xy$ is an idempotent and hence $xy=e_{xy}=e_x\cdot e_y$ as the projection $e_*:S\to E(S)$ is a homomorphism.
\smallskip

5. First we show that $xyz\in E(S)$. Since $S$ is square-linear, we get $xy\in\{x^2,y^2\}$. We lose no generality assuming that $xy=x^2$. Now consider the product $xz\in\{x^2,z^2\}$. If $xz=x^2$, then $xyz=x^2z=x(xz)=x^3\in E(S)$. If $xz=z^2$, then $xyz=x^2z=xxz=xz^2=xzz=z^2z=z^3\in E(S)$. Since the projection $e_*:S\to E(S)$ is a homomorphism, we conclude that $xyz=e_{xyz}=e_xe_ye_z$.
\smallskip

6. Assume that $x^2\notin E(S)$ but the idempotent $e_x$ is not maximal in the linear semilattice $E(S)$. Then there is an idempotent $e\in E(S)$ such that $ee_x=e_x\ne e$. It follows that
$xe\in\{x^2,e^2\}$. We claim that $xe\ne e$. Assuming that $xe=e$, we conclude that $xee=ee=e$.
On the other hand, the preceding item guarantees that $xee=e_xe_ee_e=e_xee=e_x\ne e$. So, $xe=x^2\notin E(S)$, which contradicts $xe=xee\in E(S)$.
\end{proof}

Each square-linear semigroup $S$ endowed with the set of directed edges
$$\Delta=\{(x,y)\in S\times S:xy=x^2\}$$becomes a semicomplete digraph.
In fact, the algebraic structure of a square-linear semigroup $S$ is completely determined by its digraph structure $\Delta$ and the duplication map $S\to S$, $x\mapsto x^2$. The semigroup operation $S\times S\to S$, $(x,y)\mapsto xy$, can be recovered from $\Delta$ and the duplication map  by the formula
$$xy=\begin{cases}
x^2&\mbox{if $(x,y)\in\Delta$},\\
y^2&\mbox{if $(y,x)\in\Delta$}.
\end{cases}
$$

\section{Commutativity of the semigroups $\beta(X)$, $\varphi(X)$ and $\varphi^\bullet(X)$}\label{s6}

The following characterization was proved in Theorem 4.27 of \cite{HS}.

\begin{theorem}\label{t:beta} For a commutative semigroup $X$ the following conditions are equivalent:
\begin{enumerate}
\item the semigroup $\beta(X)$ is commutative;
\item $\{a_kb_n:k,n\in\w,\;k<n\}\cap\{b_ka_n:k,n\in\w,\;k<n\}\ne\emptyset$
for any sequences $(a_n)_{n\in\w}$ and $(b_n)_{n\in\w}$ in $X$.
\end{enumerate}
\end{theorem}

\begin{corollary}\label{c:beta} If the semigroup $\beta(X)$ is commutative, then
\begin{enumerate}
\item for each square-linear subsemigroup $S\subset X$ the set $\{x^2:x\in S\}$ is finite;
\item each subsemigroup of $X$ contains a finite ideal;
\item each monogenic subsemigroup of $X$ is finite.
\end{enumerate}
\end{corollary}

\begin{proof} Assume that the semigroup $\beta(X)$ is commutative.

1. Assume that $X$ contains a square-linear subsemigroup $S\subset X$ with infinite subset $\{x^2:x\in S\}$. Then there is a sequence  $\{x_n\}_{n\in\w}$ in $S$ such that $x_n^2\ne x_m^2$ for any $n\ne m$. Define a
2-coloring $\chi:[\w]^2\to \{0,1\}$ of the set
$[\w]^2=\{(n,m)\in\w^2:n<m\}$ letting
$$\chi(n,m)=
\begin{cases}
0&\mbox{if $x_nx_m=x^2_n$}\\
1&\mbox{if $x_nx_m=x^2_m$}.
\end{cases}
$$ By Ramsey's Theorem \cite{Ramsey} (see also \cite[Theorem 5]{GRS}), there is an infinite subset $\Omega\subset\w$ and a color $c\in\{0,1\}$ such that $\chi(n,m)=c$ for any pair $(n,m)\in[\w]^2\cap\Omega^2$. Let $\Omega=\{k_n:n\in\w\}$ be the increasing enumeration of the set $\Omega$. Then for the sequences $a_n=x_{k_{2n}}$ and $b_n=x_{k_{2n+1}}$, $n\in\w$, we get
$$\{a_kb_n\}_{k<n}\cap\{b_ka_n\}_{k<n}\subset\{a_k^2\}_{k\in\w}\cap \{b_k^2\}_{k\in\w}=\{x_{k_{2n}}^2\}_{n\in\w}\cap\{x_{k_{2n+1}}^2\}_{n\in\w}=\emptyset,$$
which implies that the semigroup $\beta(X)$ is not commutative according to Theorem~\ref{t:beta}.
\smallskip

2. Let $S$ be an infinite subsemigroup of $X$. Then the semigroup $\beta(S)\subset\beta(X)$ is commutative and hence contains at most one minimal left ideal. In this case Corollary 2.23 of \cite{HLS} guarantees that the semigroup $S$ contains a finite ideal.

3. By the preceding item, for every $x\in X$ the monogenic semigroup $\{x^k\}_{k\in\IN}$ contains a finite ideal and hence is finite.
\end{proof}

\begin{theorem}\label{t:varphi} For a commutative semigroup $X$ and the semigroup $\varphi(X)$ of filters on $X$ the following conditions are equivalent:
\begin{enumerate}
\item $\varphi(X)$ is commutative;
\item $\varphi(X)$ is supercommutative;
\item $\{a_kb_n\}_{k\le n}\cap\{b_na_{n+1}\}_{n\in\w}\ne\emptyset$ for any sequences $(a_n)_{n\in\w}$ and $(b_n)_{n\in\w}$ in $X$.
\end{enumerate}
\end{theorem}

\begin{proof} We shall prove the implications $(2)\Ra(1)\Ra(3)\Ra(2)$.
The implication $(2)\Ra(1)$ is trivial.

\smallskip

$(1)\Ra(3)$ Assume that the semigroup $\varphi(X)$ is commutative and take any sequences $(a_n)_{n\in\w}$ and $(b_n)_{n\in\w}$ in $X$. Consider the filter $\A=\la A\ra$ generated by the set $A=\{a_n\}_{n\in\w}$ and the filter $\BB=\big\{B\subset X:\exists n\;\forall m\ge n\;\; b_m\in B\big\}$.
It follows that the set $C=\{a_kb_n\}_{k\le n}$ belongs to the product $\A*\BB$. Since the semigroup $\varphi(X)$ is commutative, $C\in\A*\BB=\BB*\A$ and hence there is a set $B\in\BB$ such that $BA\subset C$. By the definition of the filter $\BB$, the set $B$ contains some element $b_m$. Then $b_ma_{m+1}\in BA=AB\subset C$ and hence the intersection $\{a_kb_n\}_{k\le n}\cap\{b_na_{n+1}\}_{n\in\w}\ni b_ma_{m+1}$ is not empty.
\smallskip

$(3)\Ra(2)$ Assume that $\A*\BB\ne\A\ostar\BB$ for some filters
$\A,\BB\in\varphi(X)$. Then $\A*\BB\not\subset\A\ostar\BB$ and
some set $C\in\A*\BB$ does not belong to the filter $\A\ostar\BB$.
This means that $A*B\not\subset C$ for any sets $A\in\A$ and
$B\in\BB$. We lose no generality assuming that the set $C$ is of
the basic form $C=\bigcup_{a\in A}a*B_a$ for some
set $A\in\A$ and family $(B_a)_{a\in A}\in\BB^A$. Pick any point
$a_0\in A$ and consider the set $B_0=B_{a_0}\in\BB$. Since
$A*B_0\not\subset C$, there are points $b_0\in B_{0}$ and $a_1\in
A$ such that $a_1b_{0}\notin C$. Now consider the set $B_1=B_0\cap
B_{a_1}\in\BB$. Since $A*B_1\not\subset C$, there are points
$b_1\in B_1$ and $a_2\in A$ such that $a_2b_1\notin C$. Proceeding
by induction, for every $n\in\w$ we shall construct two sequences
of points $(a_n)_{n\in\w}$ and $(b_n)_{n\in\w}$ in $X$ such that
\begin{enumerate}
\item $a_n\in A$;
\item $b_n\in \bigcap_{i=0}^n B_{a_i}$;
\item $a_{n+1}b_n\notin C$
\end{enumerate}
for every $n\in\w$.

Observe that for each $i\le n$ we get $a_ib_n\in a_iB_{a_i}\subset C$ and hence $\{a_kb_n\}_{k\le n}\cap \{a_{n+1}b_n\}_{n\in\w}\subset C\cap (X\setminus C)=\emptyset$.
\end{proof}

\begin{proposition}\label{p6.4}
For each commutative semigroup $X$ the semigroup
$\varphi^\bullet(X)$ is supercommutative. Moreover,
$\A*\BB=\A\ostar\BB$ for each $\A\in\upsilon^\bullet(X)$,
$\BB\in\varphi(X)$.
\end{proposition}

\begin{proof} It is sufficient to prove that $\A*\BB\subset\A\ostar\BB$ for
each $\A\in\upsilon^\bullet(X)$, $\BB\in\varphi(X)$. Let
$C\in\A*\BB$. We lose no generality assuming that the set $C$ is
of the basic form $C=\bigcup_{a\in A}a*B_a$ for some finite set
$A\in\A$ and a family $(B_a)_{a\in A}\in\BB^A$. Since the set $A$
is finite, by definition of a filter, the intersection
$\bigcap_{a\in A}B_a$ is nonempty and belongs to $\BB$. Hence
$C\supset \bigcup_{a\in A}a*\big(\bigcap_{a\in A}B_a\big)=A*\big(\bigcap_{a\in
A}B_a\big)\in\A\ostar\BB$.
\end{proof}

\begin{problem} Characterize semigroups $X$ whose Stone-\v Cech extension $\beta(X)$ is supercommutative.
\end{problem}

\section{The (super)commutativity of the semigroups $\upsilon(X)$ and $\upsilon^\bullet(X)$}\label{s7}

In this section we shall characterize semigroups $X$ whose extensions $\upsilon^\bullet(X)$ and $\upsilon(X)$ are commutative or supercommutative.
The characterization will be given in terms of square-linear semigroups $X$ endowed with the digraph structure $$\Delta=\{(x,y)\in X\times X:xy=x^2\}.$$


\begin{theorem}\label{t:upsilon-bul} For a commutative semigroup $X$ the following conditions are equivalent:
\begin{enumerate}
\item the semigroup $\upsilon^\bullet(X)$ is commutative;
\item $\upsilon^\bullet(X)$ is supercommutative;
\item $\A*\BB^\perp=\BB^\perp*\A$ for any filters $\A,\BB\in\varphi^\bullet(X)\subset\upsilon^\bullet(X)$;
\item $\A*\BB=\A\ostar \BB$ for any upfamilies $\A\in\upsilon^\bullet(X)$ and $\BB\in\upsilon(X)$;
\item $\{xu,yv\}\cap\{xv,yu\}\ne\emptyset$ for any points $x,y,u,v\in X$;
\item $X$ is a square-linear semigroup whose digraph $(X,\Delta)$ contains no bipartite  cycles.
\end{enumerate}
\end{theorem}

\begin{proof} We shall prove the implications $(4)\Ra(2)\Ra(1)\Ra(3)\Ra(5)\Ra(6)\Ra(4)$ among which the implications $(4)\Ra(2)\Ra(1)\Ra(3)$ are trivial.
\smallskip

$(3)\Ra(5)$ Assume that  $\{xu,yv\}\cap\{xv,yu\}=\emptyset$ for some points $x,y,u,v\in X$, and consider the filters $\A=\la\{x,y\}\ra$ and $\BB=\la\{u,v\}\ra$, which belong to the semigroup $\varphi^\bullet(X)$. It is easy to see that $\BB^\perp=\la\{u\},\{v\}\ra$.
Observe that $\{xu,yv\}\in\A*\BB^\perp$ and $\BB^\perp*\A=\la\{ux,uy\},\{vx,vy\}\ra$. Since $\{xu,yv\}\notin \la\{ux,uy\},\{vx,vy\}\ra$, we conclude that $\A*\BB^\perp\ne\BB^\perp*\A$.
\smallskip

$(5)\Ra(6)$ To show that the semigroup $X$ is square-linear, take any two points $a,b\in X$ and put $x=v=a$ and $y=u=b$. Then $\{ab\}=\{xu,yv\}\subset \{xv,yu\}=\{a^2,b^2\}$, which means that the semigroup $X$ is square-linear. Next, we show that its digraph $(X,\Delta)$ contains no bipartite cycle.
Assuming the converse and applying Lemma~\ref{l:graph}, we conclude that $X$ contains a pure cycle $x_0,x_1,x_2,x_3,x_4$ of length 4.
For every $0\le i\le 3$ the inclusion $(x_i,x_{i+1})\in\Delta\setminus\Delta^{-1}$ implies $x_ix_{i+1}=x_{i}^2\ne x_{i+1}^2$. Since $x_4=x_0$, we get $x_4x_1=x_0x_1=x_4^2\ne x_1^2$.
Then for the points  $x=x_1,y=x_3,u=x_2,v=x_4$, we get
$$\{xu,yv\}\cap\{uy,vx\}=\{x_1x_2,x_3x_4\}\cap\{x_2x_3,x_4x_1\}=\{x_1^2,x_3^2\}\cap\{x_2^2,x_4^2\}=\emptyset.$$
So, the condition (4) does not hold.
\smallskip

$(6)\Ra(4)$ Assume that the subgroup $X$ is square-linear, but $\A*\BB\ne\A\ostar\BB$ for some upfamilies $\A\in\upsilon^\bullet(X)$ and $\BB\in\upsilon(X)$. Then $\A*\BB\not\subset\A\ostar\BB$ and hence $C\notin \A\ostar\BB$ for some set $C\in\A*\BB$.
We lose no generality assuming that $C$ is of the basic form $C=\bigcup_{a\in A}a*B_a$ for some set $A\in\A$ and sets $B_a\in\BB$, $a\in A$. Since $\A\in\upsilon^\bullet(X)$, we can assume that the set $A$ is finite.

Taking into account that $C\notin\A\ostar\BB$, we conclude that $A*B_a\not\subset C$ for each $a\in A$. Choose any element $a_0\in A$. By induction, for every $k\in\w$ we shall choose points $b_k\in B_{a_k}$ and $a_{k+1}\in A$ with $a_{k+1}*b_{k}\notin C$ as follows.
Assume that for some $k\in\w$ a point $a_k\in A$ has been constructed. Consider the set $B_{a_k}*A=A*B_{a_k}\not\subset C$ and find two points $a_{k+1}\in A$ and $b_k\in B_{a_k}$ such that $b_{k}a_{k+1}\notin C$.

Since the set $A\supset\{a_k\}_{k\in\w}$ is finite, for some point $a\in A$ the set $\Omega=\{k\in\w:a_k=a\}$ is infinite. Fix any three numbers $p,q,r\in\Omega$ such that $1<p<p+1<q<q+1<r$. Since $X$ is a square-linear semigroup, $a_qb_q\in\{a_q^2,b_q^2\}$.

Now consider two cases.
\smallskip

(i) $a_qb_q=b_q^2$. In this case we shall show that
$$(b_{q+i},a_{q+i})\in\Delta\mbox{ and }(a_{q+i+1},b_{q+i})\in\Delta$$for every $i\in\w$. This will be proved by induction on $i\in\w$. If $i=0$, then the inclusion $(b_q,a_q)\in\Delta$ follows from the equality $a_{q}b_q=b_q^2$. Assume that for some $i\in\w$ we have proved that $(b_{q+i},a_{q+i})\in\Delta$, which is equivalent to $a_{q+i}b_{q+i}=b_{q+i}^2$. It follows from $b_{q+i}^2=a_{q+i}b_{q+i}\ne b_{q+i}a_{q+i+1}\in \{b_{q+i}^2,a_{q+i+1}^2\}$ that $b_{q+i}a_{q+i+1}=a_{q+i+1}^2$ and hence $(a_{q+i+1},b_{q+i})\in\Delta$. Taking into account that
$a_{q+i+1}^2=b_{q+i}a_{q+i+1}\ne a_{q+i+1}b_{q+i+1}\in\{a_{q+i+1}^2,b_{q+i+1}^2\}$, we see that $a_{q+i+1}b_{q+i+1}=b_{q+i+1}^2$ and $(b_{q+i+1},a_{q+i+1})\in\Delta$, which completes the inductive step.

Taking into account that $\{b_{q+i}^2\}_{i\in\w}=\{a_{q+i}b_{q+i}\}_{i\in\w}\subset \{a_kb_k\}_{k\in\w}\subset C$ and $\{a_{q+i+1}^2\}_{i\in\w}=\{b_{q+i}a_{q+i+1}\}_{i\in\w}\subset \{b_ka_{k+1}\}_{k\in\w}\subset X\setminus C$,
we conclude that $\{b_{q+i}^2\}_{i\in\w}\cap\{a^2_{q+i+1}\}_{i\in\w}=\emptyset$, which implies that $(b_{q+i},a_{q+j+1})\notin\Delta\cap\Delta^{-1}$ for every $i,j\in\w$.

Now we see that $a_r,b_{r-1},a_{r-1},\dots,b_q,a_q$ is a bipartite cycle in the digraph $(X,\Delta)$.
\smallskip

(ii) $a_qb_q=a_q^2$. In this case we shall show that
$$(a_{q-i},b_{q-i})\in\Delta\mbox{ and }(b_{q-i-1},a_{q-i})\in\Delta$$for every $0\le i<q$. This will be proved by induction on $i<q$. If $i=0$, then the inclusion $(a_q,b_q)\in\Delta$ follows from the equality $a_qb_q=a_q^2$. Assume that for some non-negative number $i<q-1$ we have proved that $(a_{q-i},b_{q-i})\in\Delta$, which is equivalent to $a_{q-i}b_{q-i}=a_{q-i}^2$. It follows from $a_{q-i}^2=a_{q-i}b_{q-i}\ne b_{q-i-1}a_{q-i}\in \{b_{q-i-1}^2,a_{q-i}^2\}$ that $b_{q-i-1}a_{q-i}=b_{q-i-1}^2$ and hence $(b_{q-i-1},a_{q-i})\in\Delta$. Taking into account that
$b_{q-i-1}^2=b_{q-i-1}a_{q-i}\ne a_{q-i-1}b_{q-i-1}\in\{a_{q-i-1}^2,b_{q-i-1}^2\}$, we see that $a_{q-i-1}b_{q-i-1}=a_{q-i-1}^2$ and $(a_{q-i-1},b_{q-i-1})\in\Delta$, which completes the inductive step.

Taking into account that $\{a_{q-i}^2\}_{i=0}^{q-1}=\{a_{q-i}b_{q-i}\}_{i=0}^{q-1}\subset \{a_kb_k\}_{k\in\w}\subset C$ and $\{b_{q-i-1}^2\}_{i=0}^{q-1}=\{b_{q-i-1}a_{q-i}\}_{i=0}^{q-1}\subset \{b_ka_{k+1}\}_{k\in\w}\subset X\setminus C$,
we conclude that $\{b_{q-i-1}^2\}_{i=0}^{q-1}\cap\{a^2_{q-i}\}_{i=0}^{q-1}=\emptyset$, which implies that $(b_{q-i-1},a_{q-j})\notin\Delta\cap\Delta^{-1}$ for every $0\le i,j<q$.

Now we see that $a_p,b_p,a_{p+1},b_{p+1},\dots,a_{q-1},b_{q-1},a_q$ is a bipartite cycle in the digraph $(X,\Delta)$.
\end{proof}


\begin{theorem}\label{t:upsilon} For a commutative semigroup $X$ the following conditions are equivalent:
\begin{enumerate}
\item the semigroup $\upsilon(X)$ is commutative;
\item $\upsilon(X)$ is supercommutative;
\item the semigroups $\upsilon^\bullet(X)$ and $\beta(X)$ are commutative;
\item $\A*\BB^\perp=\BB^\perp*\A$ for any filters $\A,\BB\in\varphi(X)$;
\item $\{a_nb_n\}_{n\in\w}\cap \{b_na_{n+1}\}_{n\in\w}\ne\emptyset$ for any sequences $(a_n)_{n\in\w}$ and $(b_n)_{n\in\w}$ in $X$.
\end{enumerate}
\end{theorem}

\begin{proof} We shall prove the implications $(2)\Ra(1)\Ra(4)\Ra(5)\Ra(2)$ and $(1)\Ra(3)\Ra(5)$.
\smallskip

The implications $(2)\Ra(1)\Ra(4)$ are trivial.
\smallskip

$(4)\Ra(5)$ Assume that there are sequences $A=\{a_n\}_{n\in\w}$ and $B=\{b_n\}_{n\in\w}$ in $X$ such that $\{a_nb_n\}_{n\in\w}\cap\{b_{n}a_{n+1}\}_{n\in\w}=\emptyset$. Consider the filters $\A=\la A\ra$ and $\BB=\la B\ra$. It follows that $\{b_n\}\in \BB^\perp=\{C\subset X:C\cap B\ne\emptyset\}$ for every $n\in\w$. Assume that $\A*\BB^\perp=\BB^\perp*\A$.

Since $\{a_nb_n\}_{n\in\w}\in\A*\BB^\perp=\BB^\perp*\A$, there is $k\in\w$ such that $b_k*A\subset \{a_nb_n\}_{n\in\w}$, which is not possible as $b_ka_{k+1}\notin\{a_nb_n\}_{n\in\w}$. So, $\A*\BB^\perp\ne \BB^\perp*\A$.
\smallskip

$(5)\Ra(2)$ Assume that $\A*\BB\ne\A\ostar\BB$ for some upfamilies $\A,\BB\in\upsilon(X)$. Then $\A*\BB\not\subset\A\ostar\BB$ and hence $C\notin \A\ostar\BB$ for some set $C\in\A*\BB$.
We lose no generality assuming that $C$ is of basic form $C=\bigcup_{a\in A}aB_a$ for some set $A\in\A$ and sets $B_a\in\BB$, $a\in A$.

Taking into account that $C\notin\A\ostar\BB$, we conclude that $B_a*A=A*B_a\not\subset C$ for each $a\in A$. Choose any elements $a_0\in A$. By induction, for every $k\in\w$ we can choose points $b_k\in B_{a_k}$ and  $a_{k+1}\in A$ such that $b_{k}a_{k+1}\notin C$. Then the sequences $(a_n)_{n\in\w}$ and $(b_n)_{n\in\w}$ have the required property $\{a_nb_n\}_{n\in\w}\cap\{b_na_{n+1}\}_{n\in\w}\subset C\cap(X\setminus C)=\emptyset$, which shows that (5) does not hold.
\smallskip

The implication $(1)\Ra(3)$ is trivial.
\smallskip

$(3)\Ra(5)$. Assume that the semigroups $\beta(X)$ and $\upsilon^\bullet(X)$ are commutative but $\{a_nb_n\}_{n\in\w}\cap\{b_{n}a_{n+1}\}_{n\in\w}=\emptyset$ for some sequences $(a_n)_{n\in\w}$ and $(b_n)_{n\in\w}$. By Theorem~\ref{t:upsilon-bul}, the semigroup $X$ is square-linear and its digraph $(X,\Delta)$ contains no bipartite cycles.

 Two cases are possible.
\smallskip

(i) $a_nb_n\ne b_n^2$ for all $n\in\w$, and then $a_nb_n=a_n^2$ for all $n\in\w$. Then for each $n\in\w$ we get $\{b_n^2,a_{n+1}^2\}\ni b_na_{n+1}\notin\{a_kb_k\}_{k\in\w}=\{a_k^2\}_{k\in\w}$ and hence $b_na_{n+1}=b_{n}^2$. Then $\{a_n^2\}_{n\in\w}\cap\{b_n^2\}_{n\in\w}=\{a_nb_n\}_{n\in\w}\cap\{b_na_{n+1}\}_{n\in\w}=\emptyset$.  If for every $i<j$ we get $a_{i}b_{j}=a^2_{i}$ and $b_{i}a_{j}=b_{i}^2$, then $\{a_{i}b_{j}\}_{i<j}\cap\{b_{i}a_{j}\}_{i<j}=\emptyset$ and the semigroup $\beta(X)$ is not commutative by Theorem~\ref{t:beta}. So, there are numbers $i<j$ such that $a_{i}b_{j}\ne a^2_{i}$ or $b_{i}a_{j}\ne b_{i}^2$.

If $a_{i}b_{j}\ne a^2_{i}$, then $a_{i}b_{j}=b_{j}^2$, and $a_i,b_i,a_{i+1},b_{i+1},\dots,a_j,b_j,a_i$ if a bipartite cycle in the digraph $(X,\Delta)$, which is not possible.

If $b_{i}a_{j}\ne b_{i}^2$, then $b_{i}a_{j}=a_j^2$, and then
$b_i,a_{i+1},b_{i+1},\dots,b_{j-1},a_j,b_i$ is a bipartite cycle in the digraph $(X,\Delta)$, which is not possible.
\smallskip

(ii) $a_mb_m=b_m^2$ for some $m\in\w$. Repeating the argument of the proof of the implication $(5)\Ra(3)$ of Theorem~\ref{t:upsilon-bul}, we can check that for every $i\in\w$ \  $a_{m+i}b_{m+i}=b_{m+i}^2\ne a_{m+i+1}^2=b_{m+i}a_{m+i+1}$ and hence $\{b_{m+i}^2\}_{i\in\w}\cap\{a_{m+i+1}^2\}_{i\in\w}\subset \{a_kb_k\}_{k\in\w}\cap\{b_ka_{k+1}\}_{k\in\w}=\emptyset$. If for every $i<j$ we get $a_{m+i}b_{m+j}=b^2_{m+j}$ and $b_{m+i}a_{m+j}=a_{m+j}^2$, then $\{a_{m+i}b_{m+j}\}_{i<j}\cap\{b_{m+i}a_{m+j}\}_{i<j}=\emptyset$ and the semigroup $\beta(X)$ is not commutative by Theorem~\ref{t:beta}. So, there are numbers $i<j$ such that $a_{m+i}b_{m+j}\ne b^2_{m+j}$ or $b_{m+i}a_{m+j}\ne a_{m+j}^2$.

If $a_{m+i}b_{m+j}\ne b^2_{m+j}$, then $a_{m+i}b_{m+j}=a_{m+i}^2$, and  $a_{m+i},b_{m+j},a_{m+j},\dots,b_{m+i},a_{m+i}$ is a bipartite cycle in the digraph $(X,\Delta)$, which is not possible.

If $b_{m+i}a_{m+j}\ne a_{m+j}^2$, then $b_{m+i}a_{m+j}=b_{m+i}^2$, and $b_{m+i},a_{m+j},\dots,b_{m+i+1},a_{m+i+1},b_{m+i}$ is a bipartite cycle in the digraph $(X,\Delta)$, which is a contradiction.
\end{proof}


\section{(Super)commutativity of semigroups $N_2^\bullet(X)$ and $N_2(X)$}\label{s8}

In this section we detect semigroups with (super) commutative extensions $N_2(X)$ or $N^\bullet_2(X)$.

\begin{theorem}\label{t:N2-bul} For a commutative semigroup $X$ the following conditions are equivalent:
\begin{enumerate}
\item the semigroup $N^\bullet_2(X)$ is commutative;
\item  $N^\bullet_2(X)$ is supercommutative;
\item $\{xu,yv\}\cap\{xv,yu,xw,yw\}\ne\emptyset$ for any points $x,y,u,v,w\in X$;
\item $\A*\BB=\A\ostar \BB$ for any upfamilies $\A\in N_2^\bullet(X)$ and $\BB\in N_2(X)$;
\item $\A*\BB=\BB*\A$ for any $\A\in \varphi^\bullet(X)$ and $\BB\in N_2^\bullet(X)$;
\item Either $X$ is a square-linear semigroup whose digraph $(X,\Delta)$ contains no bipartite cycles or else $X$ contains a 2-element subgroup $H$ such that $x^3\in H$ and $xy=x^3y^3$ for each points $x,y\in X$.
\end{enumerate}
\end{theorem}

\begin{proof} We shall prove the implications $(4)\Ra(2)\Ra(1)\Ra(5)\Ra(3)\Ra(6)\Ra(4)$ among which $(4)\Ra(2)\Ra(1)\Ra(5)$ are trivial.
\smallskip

To prove that $(5)\Ra(3)$, assume that $\{xu,yv\}\cap\{xv,yu,xw,yw\}=\emptyset$ for some points $x,y,u,v,w\in X$. Consider the filter $\A=\la\{x,y\}\ra$ and the linked upfamily $\BB=\la \{u,w\},\{v,w\}\ra$. By (5), $\A*\BB=\BB*\A$. Observe that the set $\{xv,xw,yu,yw\}=x\cdot\{v,w\}\cup
y\cdot\{u,w\}$ belongs to the upfamily $\A*\BB=\BB*\A$. Then either $\{u,w\}\cdot\{x,y\}\subset\{xv,xw,yu,yw\}$ or $\{v,w\}\cdot\{x,y\}\subset\{xv,xw,yu,yw\}$. None of the inclusions is possible as $xu,yv\notin \{xv,yu,xw,yw\}$.
\smallskip

$(3)\Ra(6)$ If the semigroup $\upsilon^\bullet(X)$ is commutative,
then by Theorem~\ref{t:upsilon-bul}, $X$ is a square-linear
semigroup whose digraph $(X,\Delta)$ contains no bipartite cycles.
So, we assume that the semigroup $\upsilon^\bullet(X)$ is not
commutative. Given any element $a\in X$, put $x=v=a$, $y=u=a^2$,
and $w=a^3$. Then the condition (3) implies
$xu=yv=a^3\in\{xv,yu,xw,yw\}=\{a^2,a^4,a^5\}$, which yields
$a^3=a^5$ for each $a\in X$.  So, the semigroup $X$ is periodic
and its set of idempotents $E=\{e\in X:e^2=e\}$ is not empty. We
claim that the semilattice $E$ is linear. Assuming the converse,
find two idempotents $x,y\in E$ with $xy\notin\{x,y\}=\{x^2,y^2\}$
and put $u=x$, $v=y$, $w=xy$. Then
$\{xu,yv\}\cap\{xv,yu,xw,yw\}=\{x^2,y^2\}\cap\{xy\}=\emptyset$,
which contradicts the condition (3).

Next, we show that the semilattice $E$ has the smallest element. Assume the opposite. Since the semigroup $\upsilon^\bullet(X)$ is not commutative, Theorem~\ref{t:upsilon-bul} yields four points $x,y,u,v\in X$ such that $\{xu,yv\}\cap\{xv,yu\}=\emptyset$. Consider the projection $e_*:X\to E,$ $e_*:x\mapsto e_x$, of $X$ onto its idempotent band. Since the linear semilattice $E$ does not have the smallest idempotent, there is an idempotent $w\in E$ such that $we_{xu}=w\ne e_{xu}$ and $we_{yv}=w\ne e_{yv}$. It follows that $e_{xw}=e_x\cdot e_w=w\ne e_{xu}$ and hence $xw\ne xu$. By analogy we can prove that $\{xu,yv\}\cap\{xw,yw\}=\emptyset$, which implies $\{xu,yv\}\cap\{xv,yu,xw,yw\}=\emptyset$ and contradicts (3).

Therefore, the semilattice $E$ has the smallest element, which will be denoted by $e$. We claim that the maximal group $H_e$ containing this idempotent is not trivial. It follows from $\{xu,yv\}\cap\{xv,yu\}=\emptyset$ and $\{xu,yv\}\cap\{xv,yu,xe,ye\}\ne\emptyset\ne\{xv,yu\}\cap\{xu,yv,xe,ye\}$ that the set $\{xe,ye\}$ contains two elements and lies in the maximal subgroup $H_e$ of the idempotent $e$.
So, the group $H_e$ is not trivial. The equality $a^3=a^5$ holding for each element $a\in X$ implies that $a^2=e$ for each element $a$ of the group $H_e$. We claim that $|H_e|=2$. In the other case, we could find three pairwise distinct points $a,b,ab\in H_e\setminus\{e\}$. Put $x=u=a$, $y=v=b$, and $w=e$. Then $\{xu,yv\}\cap\{xv,yu,xw,yw\}=\{e\}\cap\{ab,a,b\}=\emptyset$, which contradicts (3).

So, $H_e=\{e,h\}$ for some element $h\in H_e$. Next, we show that $e$ is the unique element of the semilattice $E$. Assume that $E$ contains some idempotent $f\ne e$ and consider the points $x=f$, $y=h$, $u=e$, $v=h$, $w=f$. Observe that $\{xu,yv\}\cap\{xv,yu,xw,yw\}=\{fe,h^2\}\cap\{fh,he,ff,hf\}=\{e,e\}\cap\{h,f\}=\emptyset$, which contradicts (3).

Next, we check that $a^2\in H_e$ for each $a\in X$. Assume conversely that $a^2\notin H_e$. It follows from $a^3=a^5$ that $a^4$ is an idempotent which coincides with $e$ and hence $a^3\in H_e$. If $a^3=e$, then we can consider the points $x=a$, $y=h$, $u=a^2$, $v=h$ and $w=a$. Then $\{xu,yv\}\cap\{xv,yu,xw,yw\}=\{a^3,h^2\}\cap\{ah,ha^2,a^2,ha\}=\{e\}\cap\{h,a^2\}=\emptyset$, which contradicts (2). So, $a^3=h$ and then $a^{2i+1}=h$ and $a^{2i+2}=e$ for all $i\in\IN$. Consider the points $x=a$, $y=a^2$, $u=a^3$, $v=a^2$, and $w=a$. Then $\{xu,yv\}\cap\{xv,yu,xw,yw\}=\{a^4\}\cap\{a^3,a^5,a^2,a^3\}=\emptyset$, which contradicts (3).

Finally, we show that $ab\in H_e$ for any points $a,b\in X$. Assuming that $ab\notin H_e$ for some $a,b\in X$, consider the points $x=a$, $y=b$, $u=b$, $v=a$, and $w=e$. Then $\{xu,yv\}\cap\{xv,yu,xw,yw\}=\{ab\}\cap\{a^2,b^2,ae,be\}\subset \{ab\}\cap H_e=\emptyset$, which contradicts (2). So, $ab\in H_e$, and then $ab=(ab)^3=a^3b^3$.
\smallskip

$(6)\Ra(4)$ If $X$ is a square-linear semigroup whose digraph
$(X,\Delta)$ contains no bipartite cycle, then by
Theorem~\ref{t:upsilon-bul}, $\A*\BB=\A\ostar \BB$ for any
upfamilies $\A\in \upsilon^\bullet(X)$ and $\BB\in \upsilon(X)$.
Now assume that $X$ contains a two-element subgroup $H\subset X$
such that $x^3\in H$ and $xy=x^3y^3$ for any points $x,y\in X$.
This means that for the projection $\pi:X\to H$, $\pi:x\mapsto
x^3$, the semigroup $X$ is a projection extension of the subgroup
$H$. Then the semigroup $N_2(X)$ is a projection extension of the
subsemigroup $N_2(H)$. Since $|H|=2$, by Proposition~\ref{p6.4},
the semigroup $N_2(H)=\varphi^\bullet(H)$ is supercommutative and
hence for any linked upfamilies $\A,\mathcal B\in N_2(X)$ we get
$$\A*\BB=\upsilon\pi(\A)*\upsilon\pi(\mathcal B)=\upsilon\pi(\A)\ostar\upsilon\pi(\mathcal B)=\A\ostar\mathcal B.$$
\end{proof}


\begin{theorem}\label{t:N2} For a semigroup $X$ the following conditions are equivalent:
\begin{enumerate}
\item the semigroup $N_2(X)$ is commutative;
\item $N_2(X)$ is supercommutative;
\item the semigroups $N_2^\bullet(X)$ and $\beta(X)$ are commutative;
\item $\A*\BB=\A\ostar \BB$ for any upfamilies $\A\in \varphi(X)$ and $\BB\in N_2(X)$;
\item for every sequence $(a_{i})_{i\in\w}\in X^\w$ and symmetric matrix $(b_{ij})_{i,j\in\w}\subset X^{\w\times\w}$ we get\newline $\{a_{i}\cdot b_{ij}\}_{i,j\in\w}\cap \{b_{ii}\cdot a_{i+1}\}_{i\in\w}\ne\emptyset$.
\item either the semigroup $\upsilon(X)$ is commutative or else $X$ contains a 2-element subgroup $H$ such that $x^3\in H$ and $xy=x^3y^3$ for each points $x,y\in X$.
\end{enumerate}
\end{theorem}

\begin{proof} It suffices to prove the implications $(2)\Ra(1)\Ra(3)\Ra(6)\Ra(2)$ and $(2)\Ra(4)\Ra(5)\Ra(2)$. In fact, the implications $(2)\Ra(1)\Ra(3)$ and $(2)\Ra(4)$ are trivial.
\smallskip

$(3)\Ra(6)$ Assume that the semigroups $N_2^\bullet(X)$ and $\beta(X)$ are commutative but the semigroup $\upsilon(X)$ is not commutative. By Theorem~\ref{t:upsilon}, the semigroup $\upsilon^\bullet(X)$ is not commutative. Combining Theorems~\ref{t:upsilon-bul} and \ref{t:N2-bul}, we conclude that $X$  contains a 2-element subgroup $H$ such that $x^3\in H$ and $xy=x^3y^3$ for each points $x,y\in X$.
\vskip5pt

$(6)\Ra(2)$ If $\upsilon(X)$ is commutative, then by Theorem~\ref{t:upsilon}, it is supercommutative and so is its subsemigroup $N_2(X)$. If  $X$  contains a 2-element subgroup $H$ such that $x^3\in H$ and $xy=x^3y^3$ for each points $x,y\in X$, then for the projection $\pi:X\to H$, $\pi:x\mapsto x^3$, the semigroup $X$ is a projection extension of the subgroup $H$. By Proposition~\ref{p3.1},
the semigroup $N_2(X)$ is a projection extension of the subsemigroup $N_2(H)$. Since $|H|=2$, the semigroup $N_2(H)=\varphi^\bullet(H)$ is supercommutative by Proposition~\ref{p6.4}. Being a projection extension of the supercommutative semigroup $N_2(H)$, the semigroup $N_2(X)$ is supercommutative by Corollary~\ref{c:pext2}.
\smallskip

$(4)\Ra(5)$ Assume that for some sequence $(a_{i})_{i\in\w}\in X^\w$ and some symmetric matrix $(b_{ij})_{i,j\in\w}\subset X^{\w\times\w}$ we get $\{a_{i} b_{ij}\}_{i,j\in\w}\cap \{b_{ii}a_{i+1}\}_{i\in\w}=\emptyset$. Consider the filter $\A=\la A\ra\in\varphi(X)\subset N_2(X)$ generated by the set $A=\{a_i\}_{i\in\w}$ and the linked system $\mathcal B$ generated by the family $\{B_i\}_{i\in\w}$ of sets $B_i=\{b_{ij}\}_{j\in\w}$, $i\in\w$. Observe that the set $C=\{a_ib_{ij}\}_{i,j\in\w}$ belongs to $\A*\mathcal B$. Assuming that $\A*\mathcal B=\A\ostar\mathcal B$, we would find a number $i\in\w$ such that $A*B_i\subset C$, which is not possible as $a_{i+1}b_{ii}\notin C$.
\smallskip

$(5)\Ra(2)$ Assuming that $\A*\mathcal B$ is not supercommutative, we could find two linked upfamilies  $\A,\mathcal B\in N_2(X)$ such that $\A*\mathcal B\not\subset \A\ostar\mathcal B$. Then for some set $A\in\A$ and a family $(B_a)_{a\in A}\in\mathcal B^A$, we get $\bigcup_{a\in A}aB_a\notin\A\ostar \mathcal B$. It follows that for every $a\in A$ the product $A*B_a$ is not contained in the set $C=\bigcup_{a\in A}a*B_a$, which allows us to construct inductively two sequences of points $(a_i)_{i\in\w}\subset A^\w$ and $(b_i)_{i\in\w}\in X^\w$ such that $b_i\in B_{a_i}$ and $a_{i+1}b_i\notin C$ for every $i\in\w$. For every numbers $i<j$ put $b_{ii}=b_i$ and let $b_{ij}=b_{ji}$ be some point of the intersection $B_{a_i}\cap B_{a_j}$ (which is not empty by the linkedness of the upfamily $\mathcal B$).
Then the sequence $(a_i)_{i\in\w}$ and the symmetric matrix $(b_{ij})_{i,j\in\w}$ have the required property $\{a_ib_{ij}\}_{i,j\in\w}\cap\{b_{ii}a_{i+1}\}\subset C\cap (X\setminus C)=\emptyset.$
\end{proof}

\section{The commutativity of the superextension $\lambda(X)$}\label{s9}

In this section we characterize semigroups having commutative extensions $\lambda(X)$ and $\lambda^\bullet(X)$.

\begin{theorem}\label{lambda-matrix} For a commutative semigroup $X$ the following conditions are equivalent:
\begin{enumerate}
\item the semigroup $\lambda(X)$ is commutative;
\item for any symmetric matrices $(a_{ij})_{i,j\in\w}, (b_{ij})_{i,j\in\w}\in X^{\w\times\w}$ we get  $\{a_{ii}\cdot b_{ij}\}_{i,j\in\w}\cap\{b_{ii}\cdot a_{i+1,j}\}_{i,j\in\w}\ne\emptyset$.
\end{enumerate}
\end{theorem}

\begin{proof}
$(1)\Ra(2)$ Assuming that the semigroup $\lambda(X)$ is not commutative, find two maximal linked systems $\mathcal A,\mathcal B\in\lambda(X)$ such that $\mathcal A*\mathcal B\ne\mathcal B*\mathcal A$. The maximal linked upfamilies $\mathcal A*\mathcal B$ and $\mathcal B*\mathcal A$ are distinct and hence contain two disjoint sets $C\in \mathcal A*\mathcal B$ and $C'\in\mathcal B*\mathcal A$. For these sets there are sets $A\in\A$, $B\in\mathcal B$ and families of sets $(B_a)_{a\in A}\in\mathcal B^A$, $(A_b)_{b\in B}\in \mathcal A^B$ such that $\bigcup_{a\in A}aB_a\subset C$ and $\bigcup_{b\in B}bA_a\subset C'$.

By induction we can construct two sequences $\{a_{ii}\}_{i\in\w}\subset A$ and $\{b_{ii}\}_{i\in\w}$ such that $b_{ii}\in B\cap B_{a_{ii}}$ and $a_{i+1,i+1}\in A\cap A_{b_{ii}}$ for every $i\in\w$.
Since the upfamilies $\mathcal B$ and $\mathcal A$ are linked, for every numbers $i<j$ we can choose points $b_{ij}\in B_{a_{ii}}\cap B_{a_{jj}}$ and $a_{i+1,j+1}\in A_{b_{ii}}\cap A_{b_{jj}}$, and put and $b_{ji}=b_{ij}$ and $a_{j+1,i+1}=a_{i+1,j+1}$. Also put $a_{0i}=a_{i0}=a_{00}$ for all $i\in\w$. In such way we have defined two symmetric matrices $(a_{ij})_{i,j\in\w}$ and $(b_{ij})_{i,j\in\w}$ with coefficients in the semigroup $X$. Observe that for each $i,j\in\w$ we get $a_{ii}*b_{ij}\in a_{ii}*B_{a_{ii}}\subset C$ and $b_{ii}*a_{i+1,j}\in b_{ii}*A_{b_{ii}}\subset C'$, which implies that the sets
$\{a_{ii}\cdot b_{ij}\}_{i,j\in\w}$ and $\{b_{ii}\cdot a_{i+1,j}\}_{i,j\in\w}$ are disjoint.
\smallskip

$(2)\Ra(1)$ Assume that there are two symmetric matrices $(a_{ij})_{i,j\in\w},(b_{ij})_{i,j\in\w}\in X^{\w\times\w}$ such that the sets $\{a_{ii}\cdot b_{ij}\}_{i,j\in\w}$ and $\{b_{ii}\cdot a_{i+1,j}\}_{i,j\in\w}$ are disjoint. Consider the sets $A=\{a_{ii}\}_{i\in\w}$ and $A_i=\{a_{ij}\}_{j\in\w}$ which form a linked system $\{A,A_i\}_{i\in\w}$ which can be enlarged to a maximal linked system $\A$. On the other hand, the sets $B=\{b_{ii}\}_{i\in\w}$ and $B_i=\{b_{ij}\}_{j\in\w}$ form a linked upfamily, which can be enlarged to a maximal linked upfamily $\mathcal B$. We claim that $\A*\mathcal B\ne\mathcal B*\A$. This follows from the fact that the maximal linked upfamilies $\mathcal A*\mathcal B$ and $\mathcal B*\A$ contains the disjoint sets
$$\{a_{ii}b_{ij}\}_{i,j\in\w}=\bigcup_{a_{ii}\in A}a_{ii}B_i\in\mathcal A*\mathcal B$$ and
$$\{b_{ii}a_{i+1,j}\}_{i,j\in\w}=\bigcup_{b_{ii}\in B}b_{ii}A_{i+1}\in\mathcal B*\mathcal A.$$
Therefore the semigroup $\mathcal A$ is not commutative.
%
\end{proof}

For a set $X$ consider the subset $$\lambda^\bullet_3(X)=\{\A\in\lambda(X):\exists Y\subset X\mbox{  such that $|Y|\le 3$ and $\A\in\lambda(Y)\subset\lambda(X)$}\}\subset\lambda^\bullet(X).$$

\begin{theorem}\label{lambda2} For a commutative semigroup $X$ the following conditions are equivalent:
\begin{enumerate}
\item the semigroup $\lambda^\bullet(X)$ is commutative;
\item any two maximal linked systems $\A,\mathcal B\in\lambda^\bullet_3(X)$ commute;
\item any two maximal linked systems $\A\in\lambda^\bullet(X)$ and $\mathcal B\in\lambda(X)$ commute;
\item for any elements $a,b,c,x,y,z\in X$ the sets $\{ax,ay,cy,cz\}$ and
    $\{xc,xb,za,zb\}$ are not disjoint;
\item for any elements $x_0,x_1,x_2,x_3,x_4,x_5$ the sets $\{x_1x_2,x_2x_3,x_3x_4,x_4x_5\}$ and $\{x_1x_4,x_2x_5,x_0x_1,x_0x_5\}$ are not disjoint.
\end{enumerate}
\end{theorem}

\begin{proof} It suffices to prove the implications $(3)\Ra(1)\Ra(2)\Ra(4)\Leftrightarrow (5)\Ra(1)$.
In fact, the implications $(3)\Ra(1)\Ra(2)$ are trivial while the equivalence $(4)\Leftrightarrow(5)$ follows from the observation that for any points $b=x_0$, $x=x_1$, $a=x_2$, $y=x_3$, $c=x_4$, $z=x_5$ in $X$ we get $$\{ax,ay,cy,cz\}\cap\{xc,xb,za,zb\}=\{x_1x_2,x_2x_3,x_3x_4,x_4x_5\}\cap\{x_1x_4,x_1x_0,x_5x_2,x_5x_0\}.$$

$(2)\Ra(4)$ Assume that for some elements $a,b,c,x,y,z\in X$ the sets $\{ax,ay,cy,cz\}$ and $\{xc,xb,za,zb\}$ are disjoint. Consider the maximal linked systems $\A=\{A\subset X:|A\cap\{a,b,c\}|\ge 2\}$ and $\mathcal X=\{A\subset X:|A\cap\{x,y,z\}|\ge 2\}$ and observe that $\A,\mathcal X\in\lambda^\bullet_3(X)$ and the products $\mathcal A*\mathcal X$ and $\mathcal X*\A$ are distinct since they contain disjoint sets
$$a\{x,y\}\cup c\{y,z\}\in\mathcal A*\mathcal X\mbox{ and }x\{c,b\}\cup z\{a,b\}\in\mathcal X*\mathcal A.$$

$(4)\Ra(3)$ The proof of this implication is the most difficult part of the proof. Assume that $(4)$ holds but there are two non-commuting maximal linked systems $\A\in\lambda^\bullet(X)$ and $\mathcal B\in\lambda(X)$. Then the maximal linked systems $\mathcal A*\mathcal B$ and $\mathcal B*\A$ contain disjoint sets. Consequently, we can find sets $A\in\A$ and $B\in\mathcal B$ and families  $(B_a)_{a\in A}\in\mathcal B^A$ and $(A_b)_{b\in B}\in\mathcal A^B$ such that the sets $U_{\mathcal A\mathcal B}=\bigcup_{a\in A}a*B_a\in\A*\mathcal B$ and $U_{\mathcal B\A}=\bigcup_{b\in B}b*A_b\in\mathcal B*\A$ are disjoint. Since $\mathcal A\in\lambda^\bullet(X)$, we can additionally assume that the set $A$ is finite.

By analogy with the proof of Theorem~\ref{lambda-matrix}, construct inductively two sequences $(a_{i})_{i\in\w}\in A^\w$ and $(b_{i})_{i\in\w}\in B^\w$ such that $b_{i}\in B\cap B_{a_i}$ and $a_{i+1}\in A\cap A_{b_i}$.
Since the set $A$ is finite, there are two numbers $k,m$ such that $0<k<m-1$ and $a_k=a_m$.

Let $n=m-k\ge 2$ and consider the group $\IZ_n=\{0,1,\dots,n-1\}$ endowed with the group operation of addition modulo $n$, which will be denoted by the symbol $\oplus$. So, $1\oplus(n-1)=0$. For each $i\in \IZ_n$ let $a_{ii}=a_{k+i}$ and $b_{ii}=b_{k+i}$. For every numbers $i<j$ in $\IZ_n$ choose points $b_{ij}=b_{ji}\in B_{a_{ii}}\cap B_{a_{jj}}$ and $a_{ij}=a_{ji}\in A_{b_{i',i'}}\cap A_{b_{j',j'}}$ where $i',j'\in\IZ_n$ are unique numbers such that $i=i'\oplus 1$ and $j'=j\oplus 1$. It follows that $a_{ii}b_{ij}\in a_{ii}B_{a_{ii}}\subset U_{\A\mathcal B}$ and
$b_{ii}a_{i\oplus1,j}\in b_{ii}A_{b_{ii}}\in U_{\mathcal B\A}$.
So, $$\{a_{ii}*b_{ij}\}_{i,j\in\IZ_n}\cap\{b_{ii}*a_{i\oplus1,j}\}_{i,j\in\IZ_n}\subset U_{\mathcal A\mathcal B}\cap U_{\mathcal B\A}=\emptyset.$$

By induction on $i\in\IZ_n$ we shall prove that $a_{00}*b_{ii}\in U_{\A\mathcal B}$. This is trivial for $i=0$. Assume that for some positive number $i<n-1$ we have proved that $a_{00}*b_{ii}\in U_{\A\mathcal B}$.
Let $$x_0=a_{i+1,i\oplus2},\;\;x_1=b_{i,i},\;\;x_2=a_{00},\;\;x_3=b_{0,i+1},\;\;x_4=a_{i+1,i+1},\;\; x_5=b_{i+1,i+1}.$$
It follows that
$$
\begin{aligned}
\{x_1x_2,x_2x_3,x_3x_4,x_4x_5\}&=
\{b_{i,i}*a_{00},a_{00}*b_{0,i+1},b_{0,i+1}*a_{i+1,i+1},a_{i+1,i+1}*b_{i+1,i+1}\}\subset\\
&\subset U_{\A\mathcal B}\cup a_{00}*B_{a_{00}}\cup a_{i+1,i+1}*B_{a_{i+1,i+1}}\cup a_{i+1,i+1}*B_{a_{i+1,i+1}}\subset U_{\A\mathcal B}.
\end{aligned}
$$
On the other hand,
$$
\begin{aligned}
\{x_0x_1,x_0x_5,x_1x_4\}&=\{a_{i+1,i\oplus2}*b_{i,i},
a_{i+1,i\oplus 2}*b_{i+1,i+1},b_{i,i}*a_{i+1,i+1}\}\subset\\
&\subset b_{i,i}*A_{b_{i,i}}\cup b_{i+1,i+1}*A_{b_{i+1,i+1}} \cup b_{i,i}*A_{b_{i,i}}\subset U_{\mathcal B\A}.
\end{aligned}
$$
Then $\{x_1x_2,x_2x_3,x_3x_4,x_4x_5\}\cap\{x_0x_1,x_0x_5,x_1x_4\}\subset U_{\A\mathcal B}\cap U_{\mathcal B\A}=\emptyset.$
By the condition (4), the intersection
$$\{x_1x_2,x_2x_3,x_3x_4,x_4x_5\}\cap\{x_0x_1,x_0x_5,x_1x_4,x_2x_5\}$$is not empty, which implies that $a_{00}*b_{i+1,i+1}=x_2x_5\in \{x_1x_2,x_2x_3,x_3x_4,x_4x_5\}\subset U_{\A\mathcal B}$.
After completing the inductive construction, we conclude that $a_{00}*b_{n-1,n-1}\in U_{\A\mathcal B}$ which is impossible as $$a_{00}*b_{n-1,n-1}=a_k*b_{k+n-1}=a_m*b_{m-1}=b_{m-1}*a_m\in U_{\mathcal B\A}.$$
\end{proof}

We shall apply Theorem~\ref{lambda2} to detecting monogenic semigroups that have commutative superextensions.

\begin{theorem}\label{t9.3}
For a monogenic semigroup $X=\{x^k\}_{k\in\IN}$ the following conditions are equivalent
\begin{enumerate}
\item $\lambda(X)$ is commutative;
\item $\lambda^\bullet(X)$ is commutative;
\item $x^n=x^m$ for some $(n,m)\in
\{(1,2),\;(1,3),(2,3),\;(1,4),(2,4),(3,4),\;(1,5),(2,5),(3,5),(4,5),\;(2,6)\}$.
\end{enumerate}
\end{theorem}

\begin{proof} We shall prove the implications $(3)\Ra(1)\Ra (2)\Ra(3)$, among which the implication $(1)\Ra(2)$ is trivial.
\smallskip

$(3)\Ra(1)$ Assume that $x^n=x^m$ for some  $$(n,m)\in
\{(1,2),(1,3),(2,3),(1,4),(2,4),(3,4),(1,5),(2,5),(3,5),(4,5),(2,6)\}.$$
If $(n,m)\in \{(1,2),(1,3),(1,4),(1,5)\}$ then $X$ is isomorphic
to a cyclic group of order $\leq 4$ and $\lambda(X)$ is
commutative by Theorem 5.1 of \cite{BGN}.

If $(n,m)=(2,3)$, then the semigroup $\lambda(X)=X$ is commutative.

If $(n,m)\in \{(2,4),(3,4)\}$, then $|X|=3$ and $\lambda(X)=X\cup\{\triangle\}$ where $\triangle=\{A\subset X:|A|\ge 2\}$. Taking into account that $xy=yx$ and $\triangle\cdot x=x\cdot\triangle$ for all $x,y\in X$, we see that the semigroup $\lambda(X)$ is commutative.

If $(n,m)=(2,5)$, then $xa=x^4a$ for every
$a\in X$ and hence $X=\{x,x^2,x^3,x^2\}$ is a projective extension of the cyclic subgroup $\{x^2,x^3,x^4\}$. In this case the commutativity of $\lambda(X)$ follows from the commutativity of $\lambda(C_3)$ according to Proposition~\ref{c:pext2}.

By analogy, for $(n,m)=(2,6)$ the commutativity of the semigroup $\lambda(X)$ follows from the commutativity of the semigroup $\lambda(C_4)$.

Now consider the case $(n,m)=(3,5)$. In this case $X=\{x,x^2,x^3,x^4\}$ and the semigroup
$\lambda(X)$ contains 12 elements:
$$
\begin{aligned}
&k=\la\{x^k\}\ra,\\
&\triangle_k=\la \{A\subset X:|A|=2,\;x^k\notin
A\}\ra\mbox{ \  and \ }\\
&\square_k=\la\{X\setminus\{x^k\},A:A\subset
X,\;|A|=2,\;x^k\in A\}\ra,
\end{aligned}
$$where$k\in \{1,2,3,4\}$. The
following Cayley table of multiplication in the semigroup $\lambda(X)$ implies the commutativity of $\lambda(X)$:

$$
\begin{array}{|c|cccccccc|}
\hline
* & \triangle_1 & \triangle_2 & \triangle_3 & \triangle_4 & \square_1 & \square_2 & \square_3 & \square_4\\
\hline
\triangle_1 & 4 & 3 & 4 & 3 &  3 & 4 & 3 & 4 \\
\triangle_2 & 3 & \triangle_1 & 3 & \triangle_1 &  \triangle_1 & 3 & \triangle_1 & 3 \\
\triangle_3 & 4 & 3 & 4 & 3 &  3 & 4 & 3 & 4 \\
\triangle_4 & 3 & \triangle_1 & 3 & \triangle_1 &  \triangle_1 & 3 & \triangle_1 & 3 \\
\square_1 & 3 & \triangle_1 & 3 & \triangle_1 &  \triangle_1 & 3 & \triangle_1 & 3 \\
\square_2 & 4 & 3 & 4 & 3 &  3 & 4 & 3 & 4 \\
\square_3 & 3 & \triangle_1 & 3 & \triangle_1 &  \triangle_1 & 3 & \triangle_1 & 3 \\
\square_4 & 4 & 3 & 4 & 3 &  3 & 4 & 3 & 4 \\
\hline
\end{array}
$$
\smallskip

In the final case $(n,m)=(4,5)$, the product of
any two nonprincipal maximal linked upfamilies is equal to the principal
ultrafilter $\la \{x^4\}\ra$, which implies that the semigroup $\lambda(X)$ is commutative.
\smallskip

$(2)\Ra(3)$ Let $X=\{x^k\}_{k\in\IN}$ be a monogenic semigroup with commutative extension $\lambda^\bullet(X)$. If $|X|\le 4$, then $x^n=x^m$ for some $(n,m)\in  \{(1,2),(1,3),(2,3),(1,4),(2,4),(3,4),(1,5),(2,5),(3,5),(4,5)\}$. If $x^6=x^2$, then we are done. So, we assume that $x^6\ne x^2$ and $|X|\ge 5$. In this case the elements $x,x^2,x^3,x^4,x^5$ are pairwise distinct.

We claim that $x^7\in \{x^3,x^4\}$. In the opposite case we can put $x_0=x^{4}$, $x_1=x^3$,
$x_2=x$, $x_3=x^2$, $x_4=x^2$, $x_5=x$ and observe that
$$\{x_1x_2,x_2x_3,x_3x_4,x_4x_5\}\cap \{x_1x_4,x_2x_5,x_0x_1,x_0x_5\}=\{x^3, x^4\}\cap \{x^2, x^5,x^7\}=\emptyset,$$which implies that the semigroup $\lambda^\bullet(X)$ is not commutative according to Theorem~\ref{lambda2}. This contradiction shows that $x^7\in\{x^3,x^4\}$ and hence the monogenic semigroup $X$ is finite.

If $x^7=x^3$, then we can put $x_0=x^5$, $x_1=x_2=x$, $x_3=x^3$, $x_4=x_5=x^2$ and observe that
$$\{x_1x_2,x_2x_3,x_3x_4,x_4x_5\}\cap \{x_1x_4,x_2x_5,x_0x_1,x_0x_5\}=\{x^2,x^4,x^5\}\cap\{x^3,x^6,x^7\}=\emptyset$$since $x^6\ne x^2$.  By Theorem~\ref{lambda2}, the semigroup $\lambda^\bullet(X)$ is not commutative.

If $x^7=x^4$, then we can put
 $x_0=x_1=x$, $x_2=x^4$, $x_3=x^3$,
$x_4=x_5=x^2$ and observe that
$$\{x_1x_2,x_2x_3,x_3x_4,x_4x_5\}\cap \{x_1x_4,x_2x_5,x_0x_1,x_0x_5\}=\{x^5,x^7,x^4\}\cap \{x^3,x^6,x^2,x^3\}=\emptyset,$$ which implies that the semigroup $\lambda^\bullet(X)$ is not commutative according to Theorem~\ref{lambda2}.
\end{proof}

Now we establish some structural properties of semigroups $X$ having commutative superextensions $\lambda(X)$.

A semigroup $X$ is called a {\em $0$-bouquet} of its subsemigroups $X_\alpha$, $\alpha\in I$, if
\begin{itemize}
\item $X=\bigcup_{\alpha\in A}X_\alpha$;
\item $X$ has two-sided zero $0$;
\item $X_\alpha\cap X_\beta=X_\alpha*X_\beta=\{0\}$ for any distinct indices $\alpha,\beta\in I$.
\end{itemize}
In this case we write $X=\bigvee_{\alpha\in I}X_\alpha$.

\begin{proposition}\label{bouquet} Assume that a semigroup $X=\bigvee_{\alpha\in I}X_\alpha$ is a $0$-bouquet of its subsemigroups $X_\alpha$, $\alpha\in I$. The superextension $\lambda(X)$ is commutative if and only if for each $\alpha\in I$ the semigroup $\lambda(X_\alpha)$ is commutative.
\end{proposition}

\begin{proof} The ``only if'' part is trivial. To prove the ``if'' part, assume that the semigroup $\lambda(X)$ is not commutative. By Theorem~\ref{lambda-matrix}, there are two symmetric matrices $(a_{ij})_{i,j\in\w}$ and $(b_{ij})_{i,j\in\w}$ with the coefficients in $X$ such that the sets
$A=\{a_{ii}*b_{ij}\}_{i,j\in\w}$ and $B=\{b_{ii}*a_{i+1,j}\}_{i,j\in\w}$ are disjoint. Then $0\notin A$ or $0\notin B$.

First assume that $0\notin A$. Find an index $\alpha\in I$ such that $a_{00}\in X_\alpha$. It follows from $0\notin\{a_{00}b_{0j}\}_{j\in\w}$ that $b_{0j}\in X_\alpha$ for all $j\in\w$. Observe that for every $i\in\w$ we get $a_{ii}b_{i0}=a_{ii}b_{0i}\ne 0$ and hence $a_{ii}\in X_\alpha$. Finally, for each $i,j\in\w$, the inequality $a_{ii}b_{ij}\ne 0$ implies that $b_{ij}\in X_\alpha$. So, $\{a_{ii}\}_{i\in\w}\cup \{b_{ij}\}_{i,j\in A}\subset X_\alpha$. Now for every $i,j\in\w$ put
$$a_{ij}'=\begin{cases}
a_{ij}&\mbox{if $a_{ij}\in X_\alpha$},\\
0&\mbox{otherwise}
\end{cases}
$$
and observe that $(a_{ij})_{i,j\in\w}$ is a symmetric matrix with coefficients in $X_\alpha$.
It follows that $\{a_{ii}'b_{ij}\}_{i,j\in\w}=\{a_{ii}b_{ij}\}_{i,j\in\w}=A$ and $\{b_{ii}a_{i+1,j}\}_{i,j\in\w}\subset (B\cap X_\alpha)\cup\{0\}$. Since $A\cap (B\cup\{0\})=\emptyset$, Theorem~\ref{lambda-matrix} implies that the semigroup $\lambda(X_\alpha)$ is not commutative.
\smallskip

By analogy, we can treat the case $0\notin B=\{b_{ii}*a_{i+1,j}\}_{i,j\in\w}$. In this case there is $\alpha\in I$ such that $\{b_{ii}\}_{i\in\w}\cup\{a_{i+1,j}\}_{i,j\in\w}\subset X_\alpha\setminus\{0\}$.
Changing the element $a_{00}$ by $0$, if necessary, we get $\{a_{ij}\}_{i,j\in\w}\subset X_\alpha$.
Now for every $i,j\in \w$ put $$b_{ij}'=\begin{cases}
b_{ij}&\mbox{if $b_{ij}\in X_\alpha$},\\
0&\mbox{otherwise}.
\end{cases}
$$
Observe that $(a_{ij})_{i,j\in\w}$ and $(b'_{ij})_{i,j\in\w}$ are symmetric matrices with coefficients in $X_\alpha$ such that $\{a_{ii}b_{ij}'\}_{i,j\in\w}\subset A\cup\{0\}$ and $\{b'_{ii}a_{i+1,j}\}_{i,j\in\w}=(b_{ii}a_{i+1,j}\}_{i,j\in\w}=B$. Since $(A\cup\{0\})\cap B=\emptyset$, Theorem~\ref{lambda-matrix} implies that the semigroup $\lambda(X_\alpha)$ is not commutative.
\end{proof}



Now we detect regular semigroups $X$ whose superextensions $\lambda(X)$ are commutative.

In the following theorem for a natural number $n\in\IN$ by
$$C_n=\{z\in\mathbb C:z^n=1\}$$we denote the cyclic group of order $n$ and by $$L_n=\{0,\dots,n-1\}$$ the linear semilattice endowed with the operation of minimum.

For two semigroups $(X,*)$ and $(Y,\star)$ by $X\sqcup Y$ we denote the semigroup $X\times\{0\}\cup Y\times\{1\}$ endowed with the semigroup operation
$$(a,i)\circ(b,j)=\begin{cases} (a*b,0)&\mbox{if $i=0$ and $j=0$},\\
(a,0)&\mbox{if $i=0$ and $j=1$},\\
(b,0)&\mbox{if $i=1$ and $j=0$},\\
(a\star b,1)&\mbox{if $i=1$ and $j=1$}.
\end{cases}
$$
The semigroup $X\sqcup Y$ is called the {\em ordered union} of the semigroups $X$ and $Y$. For example, the ordered union $L_1\sqcup C_2$ is isomorphic to the multiplicative semigroup $\{-1,0,1\}$.

\begin{theorem}\label{t2} The superextension $\lambda(X)$ of a regular semigroup $X$ is commutative if and only if one of the following conditions holds:
\begin{itemize}
\item $X$ is isomorphic to one of the semigroups: $C_2$, $C_3$, $C_4$, $C_2\times C_2$, $C_2\times L_2$, $L_1\sqcup C_2$, $C_2\bigsqcup L_n$ for some $n\in\IN$;
\item $X=\bigvee_{\alpha\in
A}X_\alpha$ for some subsemigroups $X_\alpha$, $\alpha\in A$, isomorphic
to $L_1\sqcup C_2$ or $L_n$ for $n\in\IN$.
\end{itemize}
\end{theorem}

\begin{proof} To prove the ``if'' part, assume that a semigroup $X$ satisfies conditions (1) or (2). If $X$ is isomorphic to one of the groups $C_2$, $C_3$, $C_4$, or $C_2\times
C_2$, then its superextension $\lambda(X)$ is commutative according to Theorem 5.1 of \cite{BGN}.
If $X$ is isomorphic to $C_2\times L_2$ or $C_2\bigsqcup L_n$ for some
$n\in\IN$, then $\lambda(X)$ is commutative by Theorem 1.1 of \cite{BGi}.

Next, assume that $X=\bigvee_{\alpha\in A}X_\alpha$ is a $0$-bouquet of its subsemigroups $X_\alpha$, $\alpha\in A$, isomorphic to $L_1\sqcup C_2$ or $L_n$, $n\in\IN$. By Theorem 1.1 of \cite{BGi}, the superextension of the semigroups $L_1\sqcup C_2$ and $L_n$, $n\in\IN$, are commutative. Consequently, for every $\alpha\in X_\alpha$ the superextension $\lambda(X_\alpha)$ is commutative and by Proposition \ref{bouquet}, the superextension $\lambda(X)$ is commutative too. This completes the proof of the  ``if'' part.
\smallskip

The prove the ``only if'' part we shall use the following:

\begin{lemma}\label{l9.5} The superextension $\lambda(X)$ of a semigroup $X$ is not commutative if $X$ is isomorphic to one of the semigroups:
\begin{enumerate}
\item $L_1\sqcup C_n$ for $n\ge 3$;
\item $C_n\sqcup L_1$ for $n\ge 3$;
\item $L_1\sqcup C_2\sqcup L_1$;
\item $L_2\sqcup C_2$;
\item $(C_2\times C_2)\sqcup L_1$;
\item $L_1\sqcup (C_2\times C_2)$;
\item $C_2\sqcup C_2$.
\end{enumerate}
\end{lemma}

\begin{proof} 1. If $X=L_1\sqcup C_n=\{e_1\}\sqcup \{a^i\}_{i=0}^{n-1}$ for some $n\ge 3$, then
the maximal linked upfamilies\break
$\Box=\big\la\{e_1,a^0\},\{e_1,a\},\{e_1,a^{-1}\},\{a^0,a,a^{-1}\}\big\ra$
and $\triangle=\big\la\{a^0,a\},\{a^0,a^{-1}\},\{a,a^{-1}\}\big\ra$ do not
commute, since $\{e_1,a^0\}=a^0\{e_1,a^0\}\cup a\{e_1,
a^{-1}\}\in\triangle*\Box$ while $\{e_1,a^0\}\notin \Box*\triangle$.
\smallskip

2. If $X=C_n\sqcup L_1=\{a^i\}_{i=0}^{n-1}\sqcup \{e_2\}$ for some
$n\ge 3$, then  the maximal linked upfamilies\break
$\Box=\big\la\{a^2,a\}, \{a^2,a^0\}, \{a^2,e_2\}, \{a^0,a,e_2\}\big\ra$
and $\triangle=\big\la\{a^0,e_2\},\{a^0,a^2\},\{e_2,a^2\}\big\ra$ do not
commute, since\break $\{a^2,e_2\}=a^2\{a^0,
e_2\}\cup e_2\{a^2,e_2\}\in\Box*\triangle$ while $\{e_2,a^2\}\notin \triangle*\Box$.
\smallskip

3. If $X=L_1\sqcup C_2\sqcup L_1=\{e_1\}\sqcup\{e_2,a\}\sqcup
\{e_3\}$ where $a\ne a^2=e_2$, then the maximal linked upfamilies
$\Box_3=\la\{e_1,e_3\},\{e_2,e_3\},\{a,e_3\},\{e_1,e_2,a\}\ra$ and
$\Box_a=\la\{a,e_1\},\{a,e_2\},\{a, e_3\},\{e_1,e_2,e_3\}\ra$ do
not commute, since  $\{e_1,e_2\}=e_1\{e_1,e_2,e_3\}\cup
e_2\{e_1,e_2,e_3\}\cup a\{a,e_1\}\in\Box_3*\Box_a$ while
$\{e_1,e_2\}\notin\Box_a*\Box_3$.
\smallskip

4. If $X=L_2\sqcup C_2=\{e_1,e_2\}\sqcup \{e_3,a\}$ where $a\ne
a^2=e_3$, then the maximal linked upfamilies\break
$\Box=\la\{e_1,e_2\},\{e_1,e_3\},\{e_1,a\},\{e_2,e_3,a\}\ra$ and
$\triangle=\la\{e_2,a\},\{e_2,e_3\},\{a,e_3\}\ra$ do not commute,
since\break $\{e_2,e_3\}=e_2\{e_2,e_3\}\cup e_3\{e_2,e_3\}\cup
a\{e_2,a\}\in \Box*\triangle$ while
$\{e_2,e_3\}\notin\triangle*\Box$.


5. If $X=(C_2\times C_2)\sqcup\{e_2\}$ where $C_2\times
C_2=\{e_1,a,b,ab\}$ and $a^2=b^2=(ab)^2=e_1$, then the maximal linked
upfamilies $\Box=\la\{a,b\},\{a,e_1\},\{a,e_2\},\{e_1,e_2,b\}\ra$
and $\triangle=\la\{e_1,e_2\},\{e_1,a\},\{e_2,a\}\ra$ do not
commute, since $\{a,e_2\}=a\{e_1,
e_2\}\cup e_2\{e_2,a\}\in\Box*\triangle$ and $\{a,e_2\}\notin \triangle*\Box$.

6. If $X=\{e_1\}\sqcup (C_2\times C_2)$ where $C_2\times
C_2=\{e_2,a,b,ab\}$ and $a^2=b^2=(ab)^2=e_2$, then the maximal linked
upfamilies
$\Box=\la\{e_1,e_2\},\{e_1,a\},\{e_1,b\},\{e_2,a,b\}\ra$ and
$\triangle=\la\{e_2,a\},\{e_2,b\},\{a,b\}\ra$ do not commute,
since $\{e_1,e_2\}=e_2\{e_1,e_2\}\cup a\{e_1,
a\}\in\triangle*\Box$ and $\{e_1,e_2\}\notin \Box*\triangle$.

7. Finally assume that $X=C_2\sqcup C_2=\{e_1,a_1\}\cup\{e_2,a_2\}$ where $e_1<e_2$ are idempotents of $X$, $a_1^2=e_1$, $a_2^2=e_2$, and $e_1*a_2=e_1$. In this case the maximal linked upfamilies
$$\square_e=\big\la\{e_1,a_1\},\{e_1,a_2\},\{e_1,e_2\},\{a_1,a_2,e_2\}\big\ra\mbox{ \ and \ }
\square_a=\big\la\{a_1,e_1\},\{a_1,e_2\},\{a_1,a_2\},\{e_1,e_2,a_2\}\big\ra$$
do not commute as $\{e_1,e_2\}=e_1\{e_1,e_2\}\cup
e_2\{e_1,e_2\}\cup a_2\{e_1,a_2\}\in\square_a*\square_e $ while
$\{e_1,e_2\}\notin\square_e*\square_a$.

\end{proof}

Now we are ready to prove the ``only if'' part of Theorem~\ref{t2}. Assume that the superextension $\lambda(X)$ is commutative.
In this case the regular semigroup $X$ is commutative and consequently $X$ is a Clifford
inverse semigroup. By Theorem~5.1 of \cite{BGN}, the commutativity of $\lambda(X)$ implies that each subgroup of $X$ has cardinality $\le 4$. By Theorem 2.7 \cite{BGi}, the idempotent
band $E(X)=\{x\in X:xx=x\}$ of $X$ is a $0$-bouquet of
finite linear semilattices.

First we assume that $E(X)$ is a finite linear semilattice, which can be written as $E(X)=\{e_1,\dots,e_n\}$ for some idempotents $e_1<\dots<e_n$. For every $i\in\{1,\dots,n\}$ by $H_{e_i}$ we denote the maximal subgroup of $X$ containing the idempotent $e_i$. As we have shown the group $H_{e_i}$ has cardinality $|H_{e_i}|\le 4$.

If $n=1$, then the Clifford inverse semigroup $X$ coincides with the group $H_{e_1}$ and hence is isomorphic to $C_1=L_1$, $C_2$, $C_3$, $C_4$ or $C_2\times C_2$.

So, we assume that $n\ge 2$. Lemma~\ref{l9.5}(2,5) implies that for every $i<n$ the maximal subgroup $H_{e_i}$ has cardinality $|H_{e_i}|\le 2$.
For the maximal idempotent $e_n$ of $E(X)$ the complement $I=X\setminus H_{e_n}$ is an ideal in $X$. So, we can consider the quotient semigroup $X/I$, which is isomorphic to $L_1\sqcup H_{e_n}$. The commutativity of $\lambda(X)$ implies the commutativity of the semigroup $\lambda(X/I)$. Now Lemma~\ref{l9.5}(1,6) implies that $|H_{e_n}|\le 2$.

If $|E(X)|\ge 3$, then for any $1<i<n$, the maximal subgroup
$H_{e_i}$ is trivial according to Lemma~\ref{l9.5}(3) and then for
the maximal idempotent $e_n$, the subgroup $H_{e_n}$ is trivial
according to Lemma~\ref{l9.5}(4). Therefore, all maximal groups
$H_{e_i}$, $1<i\le n$, are trivial. If the group $H_{e_1}$ is
trivial, then $X=E(X)$ is isomorphic to the linear semilattice
$L_n$. If $H_{e_1}$ is not trivial, then $H_{e_1}$ is isomorphic
to $C_2$ and $X$ is isomorphic to $C_2\sqcup L_{n-1}$.

It remains to consider the case $|E(X)|=2$. In this case the groups $H_{e_1}$, $H_{e_2}$ have cardinality $\le 2$ and then $X$ is isomorphic to $L_2$, $C_2\sqcup L_1$, $L_1\sqcup C_2$, $C_2\times L_2$ or $C_2\sqcup C_2$. However the case $X\cong C_2\sqcup C_2$ is excluded by Lemma~\ref{l9.5}(7).
This completes the proof of the case of linear semilattice $E(X)$.
\smallskip

Now we consider the case of non-linear semilattice $E(X)$. Write $E(X)$ as a $0$-bouquet $E(X)=\bigvee_{\alpha\in I}E_\alpha$ of finite linear semilattices $E_\alpha$. Let $e_0$ be the minimal idempotent of the semilattice $E(X)$. Since $E(X)$ is not linear, there are two idempotents $e_1,e_2\in E(X)\setminus\{e_0\}$ such that $e_1e_2=e_0$. We claim that the maximal subgroup $H_{e_0}$ containing the idempotent $e_0$ is trivial. It follows from the ``linear'' case, that $|H_{e_0}|\le 2$. Assuming that $H_{e_0}$ is not trivial, write $H_{e_0}=\{a,e_0\}$ and consider the maximal linked upfamilies
$\triangle_0=\la\{a,e_1\},\{e_1,e_2\},\{a,e_2\}\ra$ and
$\triangle_a=\la\{e_0,e_1\},\{e_1,e_2\},\{e_0,e_2\}\ra$ which do not commute since $\triangle_0*\triangle_a=\la\{e_0\}\ra\ne\la\{a\}\ra=\triangle_a*\triangle_0$.
Consequently, the maximal subgroup $H_{e_0}$ is trivial and hence for every $\alpha\in A$ the subsemigroup $X_\alpha=\bigcup_{e\in E_\alpha}H_e$ is isomorphic to $L_1\sqcup C_2$ or $L_n$, $n\in\IN$, by the preceding ``linear'' case.
\end{proof}

Theorems~\ref{t9.3} and \ref{t2} imply:

\begin{corollary} If a semigroup $X$ has commutative superextension $\lambda(X)$, then
\begin{enumerate}
\item for each $x\in X$ there is a pair $(n,m)\in \{(2,5),(2,6),(3,5),(4,5)\}$ such that $x^n=x^m$;
\item the idempotent semilattice $E(X)=\{x\in X:xx=x\}$ of $X$ is
a $0$-bouquet of finite linear semilattices;
\item the regular part $R(X)=\{x\in X:x\in xXx\}$ of $X$ is isomorphic to one of the
following semigroups:
\begin{itemize}
\item $L_1$, $C_2$, $C_3$, $C_4$, $C_2\times C_2$, $C_2\times L_2$,  $C_2\bigsqcup L_n$ for some $n\in\IN$;
 \item
a 0-bouquet $\bigvee_{\alpha\in
A}X_\alpha$ of subsemigroups $X_\alpha$, $\alpha\in I$, isomorphic
to $L_1\sqcup C_2$ or $L_n$ for $n\ge 2$.
\end{itemize}
\end{enumerate}
\end{corollary}

\section{The supercommutativity of the superextensions $\lambda(X)$}\label{s10}

By Theorems~\ref{t:varphi}, \ref{t:upsilon-bul}, \ref{t:upsilon}, \ref{t:N2-bul}, \ref{t:N2}, for any semigroup $X$, the semigroups $\upsilon(X)$,
$\upsilon^\bullet(X)$, $\varphi(X)$, $\varphi^\bullet(X)$,
$N_2(X)$, $N_2^\bullet(X)$ are supercommutative if and only if
they are commutative. In contrast, the supercommutativity of the superextension $\lambda(X)$ is not equivalent to its commutativity.

\begin{theorem}\label{t10.1} For a monogenic semigroup $X=\{x^k\}_{k\in\IN}$ the following conditions are equivalent:
\begin{enumerate}
\item the semigroup $\lambda(X)$ is supercommutative;
\item the semigroup $\lambda^\bullet(X)$ is supercommutative;
\item $x^n=x^m$ for some $(n,m)\in \{(1,2),(1,3),(2,3),(2,4),(3,4),(4,5)\}$.
\end{enumerate}
\end{theorem}

\begin{proof} We shall prove the implications $(3)\Ra(1)\Ra(2)\Ra(3)$ among which the implication $(1)\Ra(2)$ is trivial.
\vskip5pt

$(3)\Ra(1)$. Assume that $x^n=x^m$ for some $(n,m)\in \{(1,2),(1,3),(2,3),(2,4),(3,4),(4,5)\}$.
For $(n,m)\in \{(1,2),(1,3),(2,3)\}$ the monogenic semigroup $X$ has cardinality $|X|\le 2$ and then the semigroup $\lambda(X)=X$ is supercommutative.

If $(n,m)=(2,4)$, then the monogenic semigroup $X$ has cardinality $|X|=3$ and for the unique non-principal maximal linked system $\triangle=\{A\subset X:|A|\ge 2\}$ in $\lambda(X)$ the product
$\triangle\ostar\triangle$ is equal to the principal ultrafilter $\la x^2\ra=\triangle*\triangle$, which implies that the semigroup $\lambda(X)$ is supercommutative.

If $(n,m)\in \{(3,4),(4,5)\}$, then any two nonprincipal maximal
linked systems $\A, \mathcal B$ contain sets $A\in\A,
B\in\mathcal B$ such that $x\notin A, x\notin B$. Then $AB$ is a
singleton, which implies $\A\ostar\mathcal B=\A*\mathcal B$. Consequently, the semigroup $\lambda(X)$ is supercommutative.
\smallskip

$(2)\Ra(3)$ Assume that for a monogenic semigroup
$X=\{x^k\}_{k\in\IN}$ the superextension $\lambda^\bullet(X)$   is
supercommutative. Then it is commutative and by
Theorem~\ref{t9.3}, $x^n=x^m$ for some pair $$(n,m)\in
\{(1,2),(1,3),(2,3),(1,4),(2,4),(3,4),(1,5),(2,5),(3,5),(4,5),(2,6)\}.$$
We claim that $|m-n|\le 2$. In the opposite case $X$ contains a
cyclic subgroup $C$ of cardinality $|C|\ge 3$. The subgroup $C$
contains an element $x\in C$ such that the points $x^{-1},x^0,x^1$
are pairwise distinct. Then for the maximal linked system
$\triangle=\la\{x^{-1},x^0\},\{x^0,x^1\},\{x^{-1},x^1\}\ra\in\lambda^\bullet(C)\subset
\lambda^\bullet(X)$ the product
$$\triangle\ostar\triangle=\big\la\{x^{-2},x^{-1},x^0\},\{x^{-1},x^0,x^{1}\},\{x^0,x^1,x^2\}\big\ra$$does not belong to $\lambda(C)$, which implies that $\triangle\ostar\triangle\ne\triangle*\triangle$ and contradicts the supercommutativity of $\lambda(X)$. So, $|m-n|\le 2$, which implies that
$(n,m)\in\{(1,2),(1,3),(2,3),(2,4),(3,4),(3,5),(4,5)\}$. It
remains to exclude the case $(n,m)=(3,5)$. In this case
$X=\{x,x^2,x^3,x^4\}$ and for the maximal linked upfamilies
$\square=\big\langle\{x^2,x^3,x^4\},\{x,x^2\},\{x,x^3\},\{x,x^4\}\big\ra$
and $\triangle=\big\la\{x,x^2\},\{x,x^3\},\{x^2,x^3\}\big\ra$ we
get $$\square \ostar\triangle=\big\langle \{x^2, x^4\},
\{x^3,x^4\}\big\rangle\neq \square*\triangle,$$ which contradicts
the supercommutativity of the semigroup $\lambda(X)$.
\end{proof}

In the following theorem by $V_3$ we denote the semilattice $\{0,1\}^2\setminus\{(1,1)\}$ endowed with the operation of coordinatewise minimum. Observe that a semilattice $X$ is isomorphic to $V_3$ if and only if $|X|=3$ and $X$ is not linear.

\begin{theorem}\label{t10.2} The superextension $\lambda(X)$ of a regular semigroup $X$ is supercommutative if and only if $X$ is isomorphic to one of the semigroups: $C_2$, $L_1\sqcup C_2$, $V_3$ or $L_n$ for $n\in\IN$.
\end{theorem}

\begin{proof} First we prove the ``if'' part of the theorem. If $X=C_2$, then its superextension $\lambda(X)=X$ is supercommutative as all maximal linked upfamilies on $X$ are principal ultrafilters.

If $X=L_1\sqcup C_2$, then $\lambda(X)$ is supercommutative since for the unique non-principal maximal linked system $\triangle=\{A\subset X:|A|\ge 2\}$ we get $\triangle\ostar\triangle=\triangle=\triangle*\triangle$.

If $X=V_3$, then $\lambda(X)$ is supercommutative since for the unique non-principal maximal linked system $\triangle=\{A\subset X:|A|\ge 2\}$ the products $\triangle\ostar\triangle=\la \min V_3\ra=\triangle*\triangle$ coincide with the principal ultrafilter generated by the minimal element $(0,0)=\min V_3$ of the semilattice $V_3$.

If $X=L_n$ for some $n\in\IN$, then the supercommutativity of the semigroup $\lambda(X)$ follows from Theorem 2.5 of \cite{BGs}.
\smallskip

To prove the ``only if'' part, assume that $X$ is a regular
semigroup with supercommutative superextension $\lambda(X)$. First
observe that every subgroup $G$ of $X$ has cardinality $|G|\le 2$.
In the opposite case the group $G$ contains an element $x\in X$
such that $|\{x^1,x^0,x^{-1}\}|=3$ where $x^0$ is the idempotent
of the group $G$. Then for the maximal linked system
$\triangle=\big\la\{x^{-1},x^0\},\{x^0,x^1\},\{x^{-1},x^1\}\big\ra$
the product $\triangle\ostar\triangle=\big\la
\{x^{-2},x^{-1},x^0\},\{x^{-1},x^0,x^1\},\{x^0,x^1,x^2\}\big\ra$
does not belong to $\lambda(X)$ and hence is not equal to
$\triangle*\triangle$. This contradiction shows that all subgroups
of $X$ has cardinality $\le 2$. This fact combined with
Theorem~\ref{t2} yields that $X$ is isomorphic to one of the
semigroups:
\begin{itemize}
\item $L_1$, $C_2$, $C_2\bigsqcup L_n$ for some $n\in\IN$;
 \item a 0-bouquet $\bigvee_{\alpha\in
I}X_\alpha$ of subsemigroups $X_\alpha$, $\alpha\in I$, isomorphic
to $L_1\sqcup C_2$ or $L_n$ for $n\ge 2$.
\end{itemize}
It remains to exclude the semigroups from this list, whose superextensions are not supercommutative.

If $X=C_2\bigsqcup L_n$, then $X$ contains the semigroup $C_2\sqcup L_1=\{e_1,a\}\cup\{e_2\}$ where $a^2=e_1\ne a$ and $e_1<e_2$ are idempotents. In this case for the maximal linked system $\triangle=\big\la\{a,e_1\},\{e_1,e_2\},\{a,e_2\}\big\ra$ we get $\triangle\ostar\triangle=\big\la\{a,e_1\},\{e_1,e_2\}\big\ra\notin\lambda(X)$ and hence $\triangle\ostar\triangle\ne\triangle*\triangle$, which means that $\lambda(X)$ is not supercommutative.

If $X=L_1\sqcup C_2=\{e_1\}\cup\{e_2,a\}$ where $a^2=e_2>e_1$, then for the maximal linked system $\triangle=\big\la\{a,e_1\},\{e_1,e_2\},\{a,e_2\}\ra$ we get $\triangle\ostar\triangle=\la\{e_1,e_2\},\{e_2,a\}\ra\notin\lambda(X)$ and hence $\triangle\ostar\triangle\ne\triangle*\triangle$, which means that $\lambda(X)$ is not supercommutative.

It remains to consider the case when $X=\bigcup_{\alpha\in I}X_\alpha$ is a $0$-bouquet of subsemigroups $X_\alpha$, $\alpha\in I$, isomorphic to $L_n$ for $n\ge 2$. If $|I|=1$, then $X$ is isomorphic to $L_n$ for some $n\ge 2$ and $\lambda(X)$ is supercommutative according to the ``if''part.

If $|I|=2$, then $X=X_i\vee X_j$ for some non-trivial linear subsemilattices $X_i,X_j\subset X$ such that $X_j*X_j=X_i\cap X_j=\{\min X\}$. If $|X_i|=|X_j|=2$, then the semilattice $X$ is isomorphic to the semilattice $V_3$ and its superextension $\lambda(X)$ is supercommutative as proved in the ``if'' part. So, we assume that $|X_i|\ge 3$ or $|X_j|\ge 3$. We loss no generality assuming that $|X_i|\ge 3$. Then we can find idempotents $e_0<e_1<e_2$ in $X_i$ and $e_3\in X_j\setminus X_i$ such that $e_1e_3=e_2e_3=e_0=\min X$. In this case for the maximal linked system $\triangle=\la\{e_1,e_2\},\{e_1,e_3\},\{e_2,e_3\}\ra$ the product $\triangle\ostar\triangle=\la\{e_0,e_1\},\{e_1,e_2\}\ra\notin\lambda(X)$ and hence $\triangle\ostar\triangle\ne\triangle*\triangle$, which means that $\lambda(X)$ is not supercommutative.

If $|I|\ge 3$, then the semigroup $X$ contains a 4-element semilattice $V_4=\{e_0,e_1,e_2,e_3\}$ where $e_ie_j=e_0=\min X$ for any distinct number $i,j\in\{1,2,3\}$. In this we can consider the maximal linked system $\triangle=\la\{e_1,e_2\},\{e_1,e_3\},\{e_2,e_3\}\ra\in\lambda(V_4)\subset\lambda(X)$ and observe that $\triangle\ostar\triangle=\la\{e_0,e_1\},\{e_0,e_2\},\{e_0,e_3\}\ra\notin\lambda(X)$. Consequently, $\triangle\ostar\triangle\ne\triangle*\triangle$ and the semigroup $\lambda(X)$ is not supercommutative.
\end{proof}

Theorems~\ref{t10.1} and \ref{t10.2} imply:

\begin{corollary} If a semigroup $X$ has supercommutative superextension $\lambda(X)$, then
\begin{enumerate}
\item for each $x\in X$ we get $x^4\in\{x^2,x^5\}$;
\item the regular
part $R(X)=\{x\in X:x\in xXx\}$ of $X$ is isomorphic to  $C_2$, $L_1\sqcup C_2$, $V_3$ or $L_n$ for some $n\in\IN$.
\end{enumerate}
\end{corollary}

\end{document}